\documentclass[11pt,a4paper]{amsart}
\usepackage{mathrsfs}
\usepackage{syntonly}
\usepackage{amsmath}
\usepackage{amsthm}
\usepackage{amsfonts}
\usepackage{color}
\usepackage{amssymb}
\usepackage{latexsym}
\usepackage{amscd,amssymb,amsopn,amsmath,amsthm,graphics,amsfonts,mathrsfs,accents,enumerate,verbatim,calc}
\usepackage[dvips]{graphicx}
\usepackage[colorlinks=true,linkcolor=red,citecolor=blue]{hyperref}
\usepackage[all]{xy}
\usepackage{enumitem}

\date{}
\allowdisplaybreaks[4] \footskip=15pt
\renewcommand{\uppercasenonmath}[1]{}

\numberwithin{equation}{section} \theoremstyle{plain}
\newtheorem*{thm*}{Main Theorem}
\newtheorem{thm}{Theorem}[section]
\newtheorem{cor}[thm]{Corollary}
\newtheorem*{cor*}{Corollary}
\newtheorem{lem}[thm]{Lemma}
\newtheorem*{lem*}{Lemma}
\newtheorem{prop}[thm]{Proposition}
\newtheorem*{prop*}{Proposition}
\newtheorem{rem}[thm]{Remark}
\newtheorem*{rem*}{Remark}
\newtheorem{exa}[thm]{Example}
\newtheorem*{exa*}{Example}
\newtheorem{df}[thm]{Definition}
\newtheorem*{df*}{Definition}

\newtheorem*{conj*}{Conjecture}

\newtheorem*{ack*}{ACKNOWLEDGEMENTS}

\newcommand{\pf}{\noindent\begin {proof}}
\newcommand{\epf}{\end{proof}}
\newcommand{\bc}{\begin{center}}
\newcommand{\ec}{\end{center}}
\newcommand{\ra}{\rightarrow}
\newcommand{\Mod}{\mbox{\rm Mod}}
\newcommand{\coker}{\mbox{\rm coker}}
\begin{document}
\begin{center}
{\Large  \bf  Special precovering classes in comma categories}

\vspace{0.5cm}  Jiangsheng Hu$^a$ and Haiyan Zhu$^b$$\footnote{Corresponding author}$ \\
\medskip

\hspace{-4mm}$^{a}$School of Mathematics and Physics, Jiangsu University of Technology,
 Changzhou 213001, China\\
  $^b$College of Science, Zhejiang University of Technology, Hangzhou 310023, China\\

E-mails: jiangshenghu@jsut.edu.cn and hyzhu@zjut.edu.cn\\
\end{center}

\bigskip
\centerline { \bf  Abstract}
 \leftskip10truemm \rightskip10truemm \noindent
 Let $T$ be a right exact functor from an abelian category $\mathscr{B}$ into another abelian category $\mathscr{A}$. Then there exists a functor ${\bf p}$ from the product category $\mathscr{A}\times\mathscr{B}$ to the comma category $(T\downarrow\mathscr{A})$.
 In this paper, we study the property of the extension closure of some classes of objects in $(T\downarrow\mathscr{A})$, the exactness of the functor ${\bf p}$ and the detailed description of orthogonal classes of a given class ${\bf p}(\mathcal{X},\mathcal{Y})$ in $(T\downarrow\mathscr{A})$. Moreover,  we characterize  when special precovering classes in abelian categories $\mathscr{A}$ and $\mathscr{B}$ can induce special precovering classes in $(T\downarrow\mathscr{A})$. As an application, we prove that under
suitable cases, the class of Gorenstein projective left $\Lambda$-modules over a triangular matrix ring $\Lambda=\left(\begin{smallmatrix}R & M \\0 & S\\\end{smallmatrix}
                                                                               \right)$
 is special precovering if and only if both the classes of Gorenstein projective left $R$-modules and left $S$-modules are special precovering. Consequently, we produce a large variety of examples of rings such that the class of Gorenstein projective modules is special precovering over them.
\medskip
\leftskip10truemm \rightskip10truemm \noindent
\hspace{1em} \\[2mm]

\leftskip0truemm \rightskip0truemm
\section { \bf Introduction}

Recall that for any abelian categories $\mathscr{A}$ and $\mathscr{B}$  and any right exact functor $T:\mathscr{B}\rightarrow\mathscr{A}$,  there exists an abelian category, denoted by $(T\downarrow\mathscr{A})$, consisting of all triples $\left(\begin{smallmatrix}A\\B\end{smallmatrix}\right)_{\varphi}$ where $\varphi: T(B)\rightarrow A$ is a morphism in $\mathscr{A}$. We note that this new abelian category is called a \emph{comma category} in \cite{FGR,Marmaridis}. Detailed definitions can be found in Definition \ref{df:comma} below.
Examples of comma categories include but are not limited to: the category of modules or complexes over a triangular matrix ring, the morphism category of an abelian category and so on (see Example \ref{e exact}). It should be noted that comma categories not only give rise to adjoint functors for constructing recollements in abelian categories and triangulated categories (see \cite{Chen,Psaroudakis}) and establish various
derived equivalences between triangular matrix algebras \cite{Ladkani}, but also are used in the study of Auslander-Reiten quivers and tilting modules (see \cite{Marmaridis}). We refer to a lecture due to  Fossum, Griffith and Reiten \cite{FGR} for more details.

Originated from the concept of injective envelopes, the approximation theory has attracted increasing interest and, hence, obtained considerable development in the context of module categories or abelian categories. Independent research by Auslander,
Reiten and Smal$\o$ in the finite-dimensional case, and by Enochs and Xu for
arbitrary modules, created a general theory of left and right approximations - or
preenvelopes and precovers - of modules. The notions of a preenvelope and a precover, tied up by a
homological notion of a complete cotorsion pair observed by Salce \cite{SL} in 1970s, are dual in the category theoretic
sense. The point is that, though there is
no duality between the categories of all modules, complete cotorsion pairs make it
possible to produce special preenvelopes once we know special precovers exist and vice versa.
Considerable energy has gone into proving that concrete classes are special precovering or
covering under suitable conditions.
Examples include the classes of modules which are flat,
Gorenstein projective and Gorenstein flat. A number of these results can be found in \cite{BEE,BHG,ET,EIK,EIO,Jorgensen,Murfet,JJarxiv}.

The main objective of this paper is to study special precovering classes in the comma category $(T\downarrow\mathscr{A})$. More precisely, we characterize when special precovering classes in abelian categories $\mathscr{A}$ and $\mathscr{B}$ can induce special precovering classes in $(T\downarrow\mathscr{A})$.

In dealing with the above problem, two technical obstacles occur. The first one is that one has to choose an appropriate special precovering class in $(T\downarrow\mathscr{A})$ from that of $\mathscr{A}$ and $\mathscr{B}$. To overcome this obstacle, we introduce a functor ${\bf p}$ from the product category $\mathscr{A}\times\mathscr{B}$ into the comma category $(T\downarrow\mathscr{A})$ and give a detailed description of orthogonal classes of a given class induced by ${\bf p}$, and then establish certain relations connecting orthogonal classes of a given subclass of $(T\downarrow\mathscr{A})$ induced by ${\bf p}$ with some corresponding classes in $\mathscr{A}$ and $\mathscr{B}$. The next technical problem we encounter is that special precovering classes are not closed under direct summands in general. For instance, the class of free $R$-modules is special precovering over any ring $R$ but it is not closed under direct summands. So a crucial ingredient for constructing complete cotorsion pairs from special precovering classes used in \cite[Lemma 2.2.6]{Gober} is missing. To circumvent this problem here, we first replace the class $\mathcal{X}$ with the class of direct summands of objects in $\mathcal{X}$, and then demonstrate that this kind of replacement can preserve the property of special precovering under certain conditions.

To state our main result more precisely, let us first introduce some definitions.

Throughout this paper, $\mathscr{A}$ and $\mathscr{B}$ are abelian categories. For any ring $R$, $\Mod R\ ({\rm mod}R)$ is the class of  (finitely generated) left $R$-modules and $\textrm{Ch}(R)$ is the class of complexes of left $R$-modules. For unexplained ones, we refer the reader to \cite{EJ,FGR,Gober}.

Let $\mathcal {X}$ be a subclass of $\mathscr{A}$. For convenience, we set
\begin{center}
{\vspace{2mm}
$\mathcal {X}^{\bot}:=\{M:\mbox{Ext}^1_\mathscr{A}(X,M)=0\mbox{ for every } X\in\mathcal {X}\}$;

\vspace{2mm}
$^{\bot}\mathcal {X}:=\{M:\mbox{Ext}^1_\mathscr{A}(M,X)=0\mbox{ for every } X\in\mathcal {X}\}$.}
\end{center}
Recall that a morphism $f : G \rightarrow M$ is called a \emph{special $\mathcal {X}$-precover} of an object $M$ if $f$ is surjective, $G\in{\mathcal{X}}$ and ${\ker} f \in{{\mathcal{X}}^{\bot}}$. Dually, a morphism $g : N \rightarrow H$ is called a \emph{special $\mathcal{X}$-preenvelope of an object $N$ if $g$ is injective, $H\in{\mathcal{X}}$ and $\textrm{coker}g \in{^{\bot}{\mathcal{X}}}$. The class $\mathcal{X}$ is called \emph{special precovering} (resp. \emph{special preenveloping}) in $\mathscr{A}$ if every object has a special $\mathcal{X}$-precover (resp. special $\mathcal{X}$-preenvelope).}

Let  $\mathcal {Y}$ be a subclass of $\mathscr{B}$. The functor $T:\mathscr{B}\rightarrow\mathscr{A}$ is called \emph{$\mathcal {Y}$-exact} if $T$ preserves the exactness of the exact sequence $0\rightarrow B\rightarrow B'\rightarrow Y\rightarrow0$ in $\mathscr{B}$ with $Y\in\mathcal {Y}$.

It is well-known that the product category $\mathscr{A}\times\mathscr{B}$ is also an abelian category whenever both $\mathscr{A}$ and $\mathscr{B}$ are abelian categories. Thus there exists a functor ${\bf p}$ from $\mathscr{A}\times\mathscr{B}$ into $(T\downarrow\mathscr{A})$. Detailed definition can be seen in Definition \ref{df:p-functor} below.
It should be noted that the functor ${\bf p}$ was used by Enochs, Cort\'{e}s-Izurdiaga and Torrecillas \cite{ECT} in the study of  Gorenstein conditions over triangular matrix rings. Let $\mathcal {X}$  be a subclass of $\mathscr{A}$ with $0\in{\mathcal{X}}$ and $\mathcal{Y}$ a subclass of $\mathscr{B}$ with $0\in{\mathcal{Y}}$. We set
\begin{displaymath}
\langle{\bf p}(\mathcal {X},\mathcal {Y})\rangle:=\{\left(\begin{smallmatrix}A\\B\end{smallmatrix}\right):0\rightarrow\left(\begin{smallmatrix}X' \\Y'\end{smallmatrix}\right)\rightarrow \left(\begin{smallmatrix}A\\B\end{smallmatrix}\right)\rightarrow\left(\begin{smallmatrix}X\\ Y \end{smallmatrix}\right)\rightarrow0 \ \mbox{is exact} \ \mbox{with} \ \left(\begin{smallmatrix}X'\\ Y' \end{smallmatrix}\right),  \left(\begin{smallmatrix}X\\ Y \end{smallmatrix}\right)  \in  {\bf p}(\mathcal {X},\mathcal {Y})\}.
\end{displaymath}
If $\mathcal {X}$ and $\mathcal {Y}$ are closed under extensions in $\mathscr{A}$ and $\mathscr{B}$, respectively, then
$\langle{\bf p}(\mathcal {X},\mathcal {Y})\rangle$ is the smallest subcategory of ($T\downarrow\mathscr{A}$) containing ${\bf p}(\mathcal {X},\mathcal {Y})$ and closed under extensions (see Proposition \ref{thm:main1}). In particular, if we choose $\mathcal {X}$ = ${\rm mod}R$ and $\mathcal {Y}$ = ${\rm mod}S$, then $\langle{\bf p}(\mathcal {X},\mathcal {Y})\rangle$ defined here is just the monomorphism category $\mathcal{M}(R, M, S)$ introduced by Xiong, Zhang and Zhang in \cite{XZZ} (see Corollary \ref{corollary:3.6}).

Let $\mathscr{A}$ be an abelian category with enough projective objects. Recall that  an object $M$ in $\mathscr{A}$ is called \emph{Gorenstein projective} if $M=\textrm{Z}^{0}(P^\bullet)$ for some exact complex $P^\bullet$ of projective objects which remains exact after applying ${\rm Hom}_{\mathscr{A}}(-,P)$ for any projective object $P$. The complex $P^\bullet$ is called a \emph{complete $\mathscr{A}$-projective resolution}. In what follows, we denote by $\mathcal{GP}_\mathscr{A}$ the subcategory of $\mathscr{A}$ consisting of Gorenstein projective objects and by $\mathcal {GP}(R)$ the class of Gorenstein projective left $R$-modules for any ring $R$.

Now, our main result can be stated as follows.

\begin{thm}\label{thm:special precovering}Let $\mathscr{A}$ and $\mathscr{B}$ both have enough projective objects and enough injective objects.
\begin{enumerate}
\item Assume that $\mathcal{X}$  is a subclass of $\mathscr{A}$ with $0\in{\mathcal{X}}$,  $\mathcal{Y}$ is a subclass of $\mathscr{B}$ with $0\in{\mathcal{Y}}$ and $T:\mathscr{B}\rightarrow \mathscr{A}$ is a $\mathcal {Y}$-exact functor. If $\mathcal {X}$ and  $\mathcal {Y}$ are special precovering, then $\langle{\bf p}(\mathcal {X},\mathcal {Y})\rangle$ is also special precovering in $(T\downarrow\mathscr{A})$. Moreover, the converse holds when $T(\mathcal {Y} \bigcap \mathcal {Y}^{\bot})\subseteq{\mathcal {X}^{\bot}}$ and $\mathcal X, \mathcal Y$ are closed under extensions.
\item If $T:\mathscr{B}\rightarrow \mathscr{A}$ is a compatible functor, then $\mathcal{GP}_\mathscr{A}$ and $\mathcal{GP}_\mathscr{B}$ are special precovering in $\mathscr{A}$ and $\mathscr{B}$, respectively if and only if $\mathcal{GP}_{(T\downarrow\mathscr{A})}$ is special precovering in $(T\downarrow\mathscr{A})$.
\end{enumerate}
\end{thm}

We note that the condition $T$ is $\mathcal{Y}$-exact in Theorem \ref{thm:special precovering}(1) can not be omitted in general (see Remark \ref{rem:3.5}). As a consequence, we refine a main result obtained by Mao in \cite[Theorem 5.6(1)]{Mao} by deleting two superfluous assumptions (see Remark \ref{rem:3.7}).

The study of the existence of special Gorenstein  projective precovers has been subject of much research in recent years. So
far the existence of special Gorenstein projective precovers (of right modules) is known over a right coherent ring for which the projective dimension of any flat left module is finite (see \cite{EIO}). Examples of such rings include but are not limited to: Gorenstein rings (see \cite{EJ}), commutative noetherian rings of finite Krull dimension (see \cite{Murfet}), as well as two sided noetherian rings $R$ such that the injective dimension of $R$ (as a left $R$-module) is finite (see \cite{EIO}). But
for arbitrary rings this is still an open question. Work on this problem can be seen in \cite{BHG,EIK,EIO,Jorgensen,Murfet} for instance.

As a direct consequence of Theorem \ref{thm:special precovering}(2) and Proposition \ref{dim-com} below,  we have the following result which gives more examples of rings (not necessary coherent) such that the class of Gorenstein projective modules is special precovering over them.

\begin{cor}\label{GP precover}Let $\Lambda=\left(
                                                                                 \begin{smallmatrix}
                                                                                   R & M \\
                                                                                   0 & S \\
                                                                                 \end{smallmatrix}
                                                                               \right)
$ be a triangular matrix ring. Assume that $_{R}M$ and $M_{S}$ have finite projective dimension and finite flat dimension, respectively. Then $\mathcal{GP}(\Lambda)$ is special precovering  in $\Mod \Lambda$ if and only if  $\mathcal{GP}(R)$ and $\mathcal{GP}(S)$ are special precovering in $\Mod R$ and $\Mod S$, respectively.
\end{cor}

The contents of this paper are arranged as follows:
In Section 2, we study the homological behavior of the functor ${\bf p}$ from  product categories into comma categories. In Section 3, we first characterize when complete hereditary cotorsion pairs in abelian categories $\mathscr{A}$ and $\mathscr{B}$ can induce complete hereditary cotorsion pairs in $(T\downarrow\mathscr{A})$ (see Proposition \ref{thm:cotorsion-pair1}). This is based on the homological behavior of the functor ${\bf p}$  established in Section 2. As a result, we give the proof of Theorem \ref{thm:special precovering}(1).
In Section 4, we first give an explicit description for an arbitrary object in the comma category $(T\downarrow\mathscr{A})$ to be Gorenstein projective (see Proposition \ref{GP}), and then give the proof of Theorem \ref{thm:special precovering}(2).

In the following sections, we always assume that $\mathscr{A}$ and $\mathscr{B}$ are abelian categories, $\mathcal{X}$ is a subclass of $\mathscr{A}$ with $0\in{\mathcal{X}}$ and $\mathcal{Y}$ is a subclass of $\mathscr{B}$ with $0\in{\mathcal{Y}}$.

\section{The  functor {\bf p} and its homological behavior}
This section is devoted to preparations for proofs of our main results in this paper.
First, we introduce a functor ${\bf p}$ from the product category $\mathscr{A}\times\mathscr{B}$ into the comma category $(T\downarrow\mathscr{A})$, and then discuss the homological behavior of the functor ${\bf p}$, including the
property of the extension closure of some classes of objects in $(T\downarrow\mathscr{A})$, the exactness of the functor ${\bf p}$ and the
detailed description of orthogonal classes of a given class ${\bf p}(\mathcal{X},\mathcal{Y})$ in $(T\downarrow\mathscr{A})$.

\begin{df}\label{df:comma}\cite{FGR,Marmaridis}
Let $T:\mathscr{B}\rightarrow\mathscr{A}$ be a right exact functor. Then the \emph{comma category }$(T\downarrow\mathscr{A})$ is defined as follows:
 \begin{enumerate}
\item The objects are triples $\left(\begin{smallmatrix}A\\B\end{smallmatrix}\right)_{\varphi}$, with $A\in{\mathcal{\mathscr{A}}}$, $B\in{\mathcal{\mathscr{B}}}$ and $\varphi: T(B)\rightarrow A$ is a morphism in $\mathscr{A}$;

\item A morphism $\left(\begin{smallmatrix}a\\b\end{smallmatrix}\right):\left(\begin{smallmatrix}A\\B\end{smallmatrix}\right)_{\varphi}\rightarrow\left(\begin{smallmatrix}A'\\B' \end{smallmatrix}\right)_{\varphi'}$ is given by two morphisms $a:A\rightarrow A'$ in $\mathscr{A}$ and $b: B\rightarrow B'$ in $\mathscr{B}$ such that $\varphi' T(b)=a\varphi$.

 \end{enumerate}
 \indent If there is no possible confusion, we sometimes omit the morphism $\varphi$.
 \end{df}

Next we give some examples of comma categories.
\begin{exa}\label{e exact}{\rm
\begin{enumerate}
\item Let $R$ and $S$ be two rings, ${_R}M_S$ an $R$-$S$-bimodule, and $\Lambda=\left(
                                                                                 \begin{smallmatrix}R & M \\0 & S\\\end{smallmatrix}
                                                                               \right)
$ the triangular matrix ring. If we define $T\cong M\otimes_{S}-: \Mod S \ra \Mod R$, then we get that $\Mod {\Lambda}$ is equivalent to the comma category $(T\downarrow \Mod R)$.

\item Let $\Lambda=\left(\begin{smallmatrix}R & M \\0 & S\\\end{smallmatrix}\right)
$ be a triangular matrix ring.  If we define  {\rm$T\cong M\otimes_{S}-: \textrm{Ch}(S) \ra \textrm{Ch}(R)$}, then {\rm$\textrm{Ch}(\Lambda)$} is equivalent to the comma category $(T\downarrow {\rm\textrm{Ch}}(R))$.

\item If $\mathscr{A}=\mathscr{B}$ and $T$ is the identity functor, then the comma category $(T\downarrow \mathscr{A})$ coincides with the
morphism category ${\rm mor}(\mathscr{A})$ of $\mathscr{A}$.

\item Let  $\mathscr{A}=$Mod$R$ and $\mathscr{B}=$Ch$(R)$. If we define $e: \mathscr{B}\rightarrow\mathscr{A}$ via $C^\bullet\mapsto C^0$ for any $C^\bullet\in{\mathscr{B}}$, then $e$ is an exact functor and we have a comma category $(e\downarrow\mathscr{A})$.

 \end{enumerate}}\end{exa}

The following functor ${\bf p}$ was introduced by Mitchell in \cite[p.29]{Mitchell}, which is a particular case of the additive left Kan extension functor. Recently, it was used by Enochs, Cort\'{e}s-Izurdiaga and Torrecillas \cite{ECT} in the study of Gorenstein conditions over triangular
matrix rings.
\begin{df} \label{df:p-functor}Let $T:\mathscr{B}\rightarrow\mathscr{A}$ be a right exact functor. Then we have the following functor:
\begin{itemize}
\item ${\bf p}:\mathscr{A}\times\mathscr{B}\rightarrow(T\downarrow\mathscr{A})$ via
${\bf p}(A,B)=\left(\begin{smallmatrix}A\oplus T(B)\\ B\end{smallmatrix}\right)_{\left(\begin{smallmatrix}0\\ 1\end{smallmatrix}\right)}$ and ${\bf p}(a,b)=\left(\begin{smallmatrix}a\oplus T(b)\\ b\end{smallmatrix}\right)$, where $(A,B)$ is an object in $\mathscr{A}\times\mathscr{B}$ and $(a,b)$ is a morphism in $\mathscr{A}\times\mathscr{B}$.

\end{itemize}
\end{df}

\begin{rem}\label{functors}{\rm (1) Let $A$ be an object in $\mathscr{A}$ and $B$ an object in $\mathscr{B}$. It is trivial to obtain that  ${\bf p}(A,B)={\bf p}(A,0)\bigoplus{\bf p}(0,B)$. Moreover, $\textbf{p}$ preserves projective objects if $\mathscr{A}$ and $\mathscr{B}$ have enough projective objects.

(2) If we define ${\bf q}:(T\downarrow\mathscr{A})\rightarrow\mathscr{A}\times\mathscr{B}$ via
${\bf q}\left(\begin{smallmatrix}A\\B \end{smallmatrix}\right)=(A,B)$ and ${\bf q}\left(\begin{smallmatrix}a\\b \end{smallmatrix}\right)=(a,b)$ for any object $(\begin{smallmatrix}A\\B \end{smallmatrix})$ in $(T\downarrow\mathscr{A})$ and any morphism $\left(\begin{smallmatrix}a\\b \end{smallmatrix}\right)$ in $(T\downarrow\mathscr{A})$, then
\textbf{p} is a left adjoint of \textbf{q}. Hence $\textbf{p}$ is a right exact functor.}
\end{rem}

Recall that a class $\mathcal{L}$ of objects in an abelian category $\mathscr{D}$ is said to  be \emph{closed under extensions} if whenever $0\ra X\ra Y\ra Z\ra0$  is exact in $\mathscr{D}$ with $X,Z\in{\mathcal{L}}$, then $Y\in{\mathcal{L}}$. For convenience, we set
\begin{displaymath}
\mathfrak{B}^{\mathcal{X}}_{\mathcal{Y}}:=\{\left(\begin{smallmatrix}X\\Y\end{smallmatrix}\right)_{\varphi}\in (T\downarrow\mathscr{A}): Y\in\mathcal {Y}, \ \varphi \ \mbox{is monic and coker}\varphi\in\mathcal {X}\}.
\end{displaymath}
Recall from the introduction that
\begin{displaymath}
\langle{\bf p}(\mathcal {X},\mathcal {Y})\rangle:=\{\left(\begin{smallmatrix}A\\B\end{smallmatrix}\right):0\rightarrow\left(\begin{smallmatrix}X' \\Y'\end{smallmatrix}\right)\rightarrow \left(\begin{smallmatrix}A\\B\end{smallmatrix}\right)\rightarrow\left(\begin{smallmatrix}X\\ Y \end{smallmatrix}\right)\rightarrow0 \ \mbox{is exact} \ \mbox{with} \ \left(\begin{smallmatrix}X'\\ Y' \end{smallmatrix}\right),  \left(\begin{smallmatrix}X\\ Y \end{smallmatrix}\right)  \in  {\bf p}(\mathcal {X},\mathcal {Y})\},
\end{displaymath}
and the functor $T:\mathscr{B}\rightarrow\mathscr{A}$ is called \emph{$\mathcal {Y}$-exact} if $T$ preserves the exactness of the exact sequence $0\rightarrow B\rightarrow B'\rightarrow Y\rightarrow0$ in $\mathscr{B}$ with $Y\in\mathcal {Y}$.

\begin{prop}\label{thm:main1} Consider the following conditions:
\begin{enumerate}
 \item $\mathcal {X}$ and $\mathcal {Y}$ are closed under extensions in $\mathscr{A}$ and $\mathscr{B}$, respectively.
  \item $\langle{\bf p}(\mathcal {X},\mathcal {Y})\rangle=\mathfrak{B}^{\mathcal{X}}_{\mathcal{Y}}$ is the smallest subclass of ($T\downarrow\mathscr{A}$) containing ${\bf p}(\mathcal {X},\mathcal {Y})$ and closed under extensions.

 \item $\mathfrak{B}^{\mathcal{X}}_{\mathcal{Y}}$ is closed under extensions.
 \end{enumerate}
Then $(1)\Rightarrow (2)\Rightarrow (3)$. The converses hold if $T:\mathscr{B}\rightarrow\mathscr{A}$ is $\mathcal{Y}$-exact.
\end{prop}
\begin{proof}
$(1)\Rightarrow (2)$. Let $\left(\begin{smallmatrix}A\\B\end{smallmatrix}\right)_{\varphi}\in\langle{\bf p}(\mathcal {X},\mathcal {Y})\rangle$. Then there is an exact sequence
$$0\rightarrow\left(\begin{smallmatrix}X'\oplus T(Y')\\Y'\end{smallmatrix}\right)\rightarrow \left(\begin{smallmatrix}A\\B\end{smallmatrix}\right)\rightarrow\left(\begin{smallmatrix}X\oplus T(Y)\\ Y \end{smallmatrix}\right)\rightarrow0 $$ in $(T\downarrow\mathscr{A})$ with $X, X'\in\mathcal{X}$ and $Y, Y'\in\mathcal{Y}$. Thus the sequence  $0\rightarrow Y'\rightarrow B\rightarrow Y\rightarrow0$ is exact in $\mathscr{B}$,  and hence we have the following commutative diagram with exact rows:
$$\xymatrix{ &T(Y')\ar[r]\ar[d]&T(B)\ar[r]\ar[d]^\varphi&T(Y)\ar[r]\ar[d]&0\\
    0\ar[r]&X'\oplus T(Y') \ar[r]&A\ar[r]&X\oplus T(Y) \ar[r]&0.}$$
Clearly, $\varphi$ is monic and  $B\in \mathcal{Y}$ since $\mathcal{Y}$ is closed under extensions. Moreover, by the Snake Lemma, we have an exact sequence $0\rightarrow X'\rightarrow{\rm coker}\varphi\rightarrow X\rightarrow0$ which implies ${\rm coker}\varphi\in \mathcal{X}$, as desired.

Conversely, assume that  $\left(\begin{smallmatrix}X\\Y\end{smallmatrix}\right)_{\varphi}$ is an object in $(T\downarrow\mathscr{A})$ such that $Y\in \mathcal{Y},\varphi \mbox{ is monic and coker}\varphi\in \mathcal{X}$. So $\left(\begin{smallmatrix}X\\Y\end{smallmatrix}\right)_{\varphi}\in\langle{\bf p}(\mathcal{X},\mathcal{Y})\rangle$ follows from the exact sequence $$0\rightarrow\left(\begin{smallmatrix}T(Y)\\Y \end{smallmatrix}\right)\rightarrow \left(\begin{smallmatrix}X\\Y\end{smallmatrix}\right)\rightarrow\left(\begin{smallmatrix}{\rm coker}\varphi \\0\end{smallmatrix}\right)\rightarrow0 $$ in $(T\downarrow\mathscr{A})$.

By the proof above, one can get that $\langle{\bf p}(\mathcal {X},\mathcal {Y})\rangle$ is closed under extensions. So $\langle{\bf p}(\mathcal {X},\mathcal {Y})\rangle$ is the smallest subclass of ($T\downarrow\mathscr{A}$) containing ${\bf p}(\mathcal {X},\mathcal {Y})$ and closed under extensions.

 $(2)\Rightarrow (3)$ is trivial.

  $(3)\Rightarrow (1)$.
 Let $0\ra A'\ra A\ra A''\ra0$ be an exact sequence in $\mathscr{A}$ with $A',A''\in{\mathcal{X}}$. Then $0\rightarrow\left(\begin{smallmatrix}A'\\0 \end{smallmatrix}\right)\rightarrow \left(\begin{smallmatrix}A\\0\end{smallmatrix}\right)\rightarrow\left(\begin{smallmatrix} A''\\0\end{smallmatrix}\right)\rightarrow0 $ is an exact sequence in $(T\downarrow\mathscr{A})$ with $\left(\begin{smallmatrix}A'\\0\end{smallmatrix}\right)$, $\left(\begin{smallmatrix}A''\\0\end{smallmatrix}\right)\in{ \mathfrak{B}^{\mathcal{X}}_{\mathcal{Y}}}$. Thus $\left(\begin{smallmatrix}A\\0\end{smallmatrix}\right)\in{ \mathfrak{B}^{\mathcal{X}}_{\mathcal{Y}}}$, and hence we get that $A\in{\mathcal{X}}$ by hypothesis. So $\mathcal{X}$ is closed under extensions, as desired. On the other hand, let $0\ra B'\ra B\ra B''\ra0$ be an exact sequence in $\mathscr{B}$ with $B',B''\in{\mathcal{Y}}$. Note that $T$ is $\mathcal{Y}$-exact by hypothesis. Then $0\rightarrow\left(\begin{smallmatrix}T(B')\\B' \end{smallmatrix}\right)\rightarrow \left(\begin{smallmatrix}T(B)\\B\end{smallmatrix}\right)\rightarrow\left(\begin{smallmatrix} T(B'')\\B''\end{smallmatrix}\right)\rightarrow0 $ is an exact sequence in $(T\downarrow\mathscr{A})$. Thus we can get that $\left(\begin{smallmatrix}T(B')\\B'\end{smallmatrix}\right)$, $\left(\begin{smallmatrix}T(B'')\\B''\end{smallmatrix}\right)\in \mathfrak{B}^{\mathcal{X}}_{\mathcal{Y}}$ by hypothesis. It follows that $\left(\begin{smallmatrix}T(B)\\B\end{smallmatrix}\right)\in  \mathfrak{B}^{\mathcal{X}}_{\mathcal{Y}}$, which implies that $B\in{\mathcal{Y}}$. This completes the proof.
\end{proof}

\begin{lem}\label{lem:2.2} Let $\Lambda=\left( \begin{smallmatrix} R & M \\0 & S \\\end{smallmatrix} \right)$ be a triangular matrix ring.
\begin{enumerate}
\item {\rm (\cite[Theorem 3.1]{HV2})} A left $\Lambda$-module $\left(\begin{smallmatrix}{X}\\  {Y} \end{smallmatrix}\right)_\varphi$ is a projective left $\Lambda$-module if and only if $Y$ is a projective left $S$-module and $\varphi:M\otimes_S Y\rightarrow X$ is an injective $R$-morphism with a projective cokernel.
\item {\rm (\cite[Proposition 1.14]{FGR})} A left $\Lambda$-module $\left(\begin{smallmatrix}{X}\\  {Y} \end{smallmatrix}\right)_\varphi$ is a flat left $\Lambda$-module if and only if $Y$ is a flat left $S$-module and $\varphi:M\otimes_S Y\rightarrow X$ is an injective $R$-morphism with a flat cokernel.
\end{enumerate}
\end{lem}
For any ring $R$, the classes of projective and flat left $R$-modules will be denoted by $\mathcal{P}(R)$ and $\mathcal{F}(R)$, respectively. By Proposition \ref{thm:main1} and Lemma \ref{lem:2.2}, we have the following corollary.

\begin{cor}\label{corollary:3.5} Let $\Lambda=\left(
                                                                                 \begin{smallmatrix}
                                                                                   R & M \\
                                                                                   0 & S \\
                                                                                 \end{smallmatrix}
                                                                               \right)
$ be a triangular matrix ring. Then
\begin{enumerate}
\item $\mathcal{F}(\Lambda)\cong \langle{\bf p}({\mathcal{F}(R)},{\mathcal{F}(S)})\rangle$.
\item $\mathcal{P}(\Lambda)\cong \langle{\bf p}({\mathcal{P}(R)},{\mathcal{P}(S)})\rangle$.
\end{enumerate}
\end{cor}

Let ${_R}M_S$ be a finitely generated $R$-$S$-bimodule over a pair of Artin algebras $R$ and $S$, and $\Lambda=\left(
                                                                                 \begin{smallmatrix}
                                                                                   R & M \\
                                                                                   0 & S \\
                                                                                 \end{smallmatrix}
                                                                               \right)
$ the triangular matrix algebra. Recall from \cite{XZZ}
that the \emph{monomorphism category} $\mathcal{M}(R, M, S)$ induced by a bimodule ${_R}M_S$ is the subcategory of finitely generated left $\Lambda$-modules consisting of $\left(\begin{smallmatrix}{X}\\  {Y} \end{smallmatrix}\right)_\varphi$ such that $\varphi:M\otimes_S Y\rightarrow X$ is an injective $R$-morphism. When ${_R}M_S$ =${_R}R_R$, it is the classical submodule category $\mathscr{S}(R)$ in \cite{RS1,RS2,RS3}.

\begin{cor} \label{corollary:3.6} Let $R$ and $S$ be Artin algebras, and $\Lambda=\left(
                                                                                 \begin{smallmatrix}
                                                                                   R & M \\
                                                                                   0 & S \\
                                                                                 \end{smallmatrix}
                                                                               \right)
$ a triangular matrix algebra. Then $\mathcal{M}(R, M, S)\cong \langle{\bf p}({\rm mod}R, {\rm mod}S)\rangle$.
\end{cor}

The following proposition characterizes when the functor ${\bf p}$ is exact.

\begin{prop}\label{prop:3.6} The following are true for any right exact functor $T:\mathscr{B}\rightarrow\mathscr{A}$:
 \begin{enumerate}
\item $T$ is $\mathcal {Y}$-exact if and only if two exact sequences $0\rightarrow A'\rightarrow A\rightarrow A''\rightarrow0$ in $\mathscr{A}$ and $0\rightarrow B'\rightarrow B\rightarrow Y\rightarrow0$ in $\mathscr{B}$ with $Y\in{\mathcal{Y}}$ induce an exact sequence $0\rightarrow{\bf p}(A',B')\rightarrow{\bf p}(A,B)\rightarrow{\bf p}(A'',Y)\rightarrow0 $ in $(T\downarrow\mathscr{A})$.

 \item $T$ is exact if and only if ${\bf p}:\mathscr{A}\times\mathscr{B}\rightarrow(T\downarrow\mathscr{A})$ is exact.
 \end{enumerate}
\end{prop}
\begin{proof} We only need to prove (1), because (2) is a direct consequence of (1). Assume that $T:\mathscr{B}\rightarrow\mathscr{A}$ is $\mathcal {Y}$-exact. Let $0\rightarrow A'\rightarrow A\rightarrow A''\rightarrow0$ be an exact sequence in $\mathscr{A}$ and $0\rightarrow B'\rightarrow B\rightarrow Y\rightarrow0$ an exact sequence in $\mathscr{B}$ with $Y\in{\mathcal{Y}}$. Then $0\rightarrow T(B')\rightarrow T(B)\rightarrow T(Y)\rightarrow0$ is an exact sequence in $\mathscr{A}$. So the sequence $0\rightarrow{\bf p}(A',B')\rightarrow{\bf p}(A,B)\rightarrow{\bf p}(A'',Y)\rightarrow0 $ is exact in $(T\downarrow\mathscr{A})$.

Conversely, let $0\rightarrow B'\rightarrow B\rightarrow Y\rightarrow0$ be an exact sequence in $\mathscr{B}$ with $Y\in\mathcal {Y}$. Note that $0\rightarrow{\bf p}(0,B')\rightarrow{\bf p}(0,B)\rightarrow{\bf p}(0,Y)\rightarrow0 $ is an exact sequence in $(T\downarrow\mathscr{A})$ by hypothesis. So $0\rightarrow T(B')\rightarrow T(B)\rightarrow T(Y)\rightarrow0$ is an exact sequence in $\mathscr{A}$. This completes the proof.
\end{proof}

\begin{prop}\label{thm:main3} If $T:\mathscr{B}\rightarrow\mathscr{A}$ is $\mathcal {Y}$-exact, then  $\langle{\bf p}(\mathcal {X},\mathcal {Y})\rangle^{\bot}=\left(\begin{smallmatrix}\mathcal{X}^{\bot}\\\mathcal {Y}^{\bot}\end{smallmatrix}\right)$ holds in the category $(T\downarrow\mathscr{A})$.
\end{prop}

\begin{proof} In the sequel, we need the following identities $$\langle{\bf p}(\mathcal {X},\mathcal {Y})\rangle^{\bot}={\bf p}(\mathcal {X},\mathcal {Y})^{\bot}={\bf p}(\mathcal {X},0)^{\bot}\bigcap{\bf p}(0,\mathcal {Y})^{\bot},$$ which hold by Remarks \ref{functors}(1).

At first, we claim that $\left(\begin{smallmatrix}\mathcal{X}^{\bot}\\ \mathcal {Y}^{\bot} \end{smallmatrix}\right)\subseteq{\bf p}(\mathcal {X},\mathcal {Y})^{\bot}$. Let $\left(\begin{smallmatrix}A\\B \end{smallmatrix}\right)\in \left(\begin{smallmatrix}\mathcal{X}^{\bot}\\ \mathcal {Y}^{\bot} \end{smallmatrix}\right)$, it is sufficient to show the following exact sequences $$\zeta:0\rightarrow\left(\begin{smallmatrix}A\\B \end{smallmatrix}\right){\rightarrow} \left(\begin{smallmatrix}M\\N\end{smallmatrix}\right)_\varphi\stackrel{{\tiny\left(\begin{smallmatrix}m\\n\end{smallmatrix}\right)}}
{\rightarrow}\left(\begin{smallmatrix}T(Y)\\Y\end{smallmatrix}\right)\rightarrow0$$
$$\xi:0\rightarrow\left(\begin{smallmatrix}A\\B \end{smallmatrix}\right){\rightarrow} \left(\begin{smallmatrix}D\\B\end{smallmatrix}\right){\rightarrow}\left(\begin{smallmatrix}X\\0\end{smallmatrix}\right)\rightarrow0$$ are split for any $X\in{\mathcal{X}}$ and $Y\in{\mathcal{Y}}$. In fact, for $\zeta$, we have $g:Y\rightarrow N$ such that $ng=1_Y$ by hypothesis. It follows that $m (\varphi T(g))=T(n) T(g)=T(ng)=1_{T(Y)}$. That is $\left(\begin{smallmatrix}m\\n\end{smallmatrix}\right)\left(\begin{smallmatrix}\varphi T(g)\\g\end{smallmatrix}\right)=1$  which implies that the sequence $\zeta$ is split, as desired. And $\xi=0$ since the exact sequence $0\rightarrow A\rightarrow D\rightarrow X\rightarrow0$ is split.

For the reverse containment $\langle{\bf p}(\mathcal {X},\mathcal {Y})\rangle^{\bot} \subseteq \left(\begin{smallmatrix}\mathcal{X}^{\bot}\\ \mathcal {Y}^{\bot} \end{smallmatrix}\right)$, we only need to show that ${\bf p}(\mathcal {X},0)^{\bot}\subseteq\left(\begin{smallmatrix}\mathcal {X}^{\bot}\\\mathscr{B} \end{smallmatrix}\right)$ and ${\bf p}(0,\mathcal {Y})^{\bot}\subseteq\left(\begin{smallmatrix}\mathscr{A}\\\mathcal {Y}^{\bot}\end{smallmatrix}\right)$.
Let $\left(\begin{smallmatrix}A\\B\end{smallmatrix}\right)$ be an object in ${\bf p}(\mathcal {X},0)^{\bot}$. Assume that $\varepsilon:0\rightarrow A\rightarrow D\rightarrow X\rightarrow0$ is an exact sequence in $\mathscr{A}$ with $X\in\mathcal{X}$. Then the following exact sequence $$0\rightarrow\left(\begin{smallmatrix}A\\B \end{smallmatrix}\right){\rightarrow} \left(\begin{smallmatrix}D\\B\end{smallmatrix}\right)\rightarrow\left(\begin{smallmatrix}X\\0\end{smallmatrix}\right)\rightarrow0$$
 is split in $(T\downarrow\mathscr{A})$ by hypothesis. It follows that the sequence $\varepsilon$ is split, and hence $A\in {\mathcal{X}}^{\bot}$. So ${\bf p}(\mathcal {X},0)^{\bot}\subseteq\left(\begin{smallmatrix}\mathcal {X}^{\bot}\\\mathscr{B} \end{smallmatrix}\right)$, as desired.

 On the other hand, let $\left(\begin{smallmatrix}A\\B\end{smallmatrix}\right)$ be an object in ${\bf p}(0,\mathcal {Y})^{\bot}$ and $\varsigma:0\rightarrow B\rightarrow N\rightarrow Y\rightarrow0$ an exact sequence in $\mathscr{B}$ with $Y\in\mathcal{Y}$. Note that $T$ is $\mathcal {Y}$-exact. It follows that $0\rightarrow T(B)\rightarrow T(N)\rightarrow T(Y)\rightarrow0$ is an exact sequence in $\mathscr{A}$. Consider the following pushout diagram:
 $$\xymatrix{0\ar[r]&T(B)\ar[r]\ar[d]&T(N)\ar[r]\ar[d]&T(Y)\ar[r]\ar@{=}[d]&0\\
    0\ar[r]&A \ar[r]&L\ar[r]&T(Y) \ar[r]&0.}$$
 Thus we have an exact sequence $\varrho:0\rightarrow\left(\begin{smallmatrix}A\\B \end{smallmatrix}\right){\rightarrow} \left(\begin{smallmatrix}L\\N\end{smallmatrix}\right)\rightarrow\left(\begin{smallmatrix}T(Y)\\Y\end{smallmatrix}\right)\rightarrow0$ in $(T\downarrow\mathscr{A})$. Hence the sequence $\varrho$ is split by hypothesis, and so the sequence $\varsigma$ is split, i.e. $B\in\mathcal{Y}^{\bot}$. Consequently, ${\bf p}(0,\mathcal {Y})^{\bot}\subseteq\left(\begin{smallmatrix}\mathscr{A}\\\mathcal {Y}^{\bot}\end{smallmatrix}\right)$. This completes the proof.
\end{proof}

\begin{lem}\label{orth5} If $\mathscr{A}$ has enough injective objects, then $\langle{\bf p}({\mathscr {A}},{\mathscr {B}})\rangle={^{\bot}\left(\begin{smallmatrix}\mathcal {I}\\0 \end{smallmatrix}\right)}$, where $\mathcal {I}$ is the class of injective objects in $\mathscr{A}$.
\end{lem}

\pf Assume that $\mathscr{A}$ has enough injective objects and $\mathcal{I}$ is the class of injective objects in $\mathscr{A}$. Let $\left(\begin{smallmatrix}A\\B\end{smallmatrix}\right)_\varphi\in{^{\bot}\left(\begin{smallmatrix}\mathcal {I}\\0 \end{smallmatrix}\right)}$.
Then we have a monomorphism  $\sigma:T(B)\rightarrow I$ with $I$ injective. Thus we have an exact sequence $\xi:0\rightarrow\left(\begin{smallmatrix}I\\0\end{smallmatrix}\right){\rightarrow} \left(\begin{smallmatrix}I\oplus A\\B\end{smallmatrix}\right)\rightarrow\left(\begin{smallmatrix}A\\B\end{smallmatrix}\right)\rightarrow0$ in $(T\downarrow\mathscr{A})$  which is induced from the exact commutative diagram
$$\xymatrix{0\ar[r]&0\ar[r]\ar[d]&T(B)\ar@{=}[r]\ar[d]^{\left(\begin{smallmatrix}\sigma\\ \varphi\end{smallmatrix}\right)}&T(B)\ar[r]\ar[d]&0\\
    0\ar[r]&I \ar[r]^{\left(\begin{smallmatrix}1\\ 0\end{smallmatrix}\right)}&I\oplus A\ar[r]^{(0,1)}&A \ar[r]&0.}$$
By hypothesis, $\xi$ is split. It follows that there is a morphism $(f,g):I\oplus A\rightarrow I$ such that $(f,g)\left(\begin{smallmatrix}1\\0\end{smallmatrix}\right)=1_I$ and $(f,g)\left(\begin{smallmatrix}\sigma\\\varphi\end{smallmatrix}\right)=0$ which implies $g\varphi=-\sigma$. So $\varphi$ is monic, that is $\left(\begin{smallmatrix}A\\B\end{smallmatrix}\right)_\varphi\in\langle{\bf p}({\mathscr {A}},{\mathscr {B}})\rangle$ by Proposition \ref{thm:main1}.

Conversely, assume that $\left(\begin{smallmatrix}A\\B\end{smallmatrix}\right)_\varphi\in\langle{\bf p}({\mathscr {A}},{\mathscr {B}})\rangle$. Consider the exact sequence $$\xi:0\rightarrow\left(\begin{smallmatrix}I\\0\end{smallmatrix}\right)\stackrel{{\tiny\left(\begin{smallmatrix}i\\0\end{smallmatrix}\right)}}{\rightarrow} \left(\begin{smallmatrix}M\\B\end{smallmatrix}\right)_\phi\stackrel{{\tiny\left(\begin{smallmatrix}\pi\\1\end{smallmatrix}\right)}}
\rightarrow\left(\begin{smallmatrix}A\\B\end{smallmatrix}\right)\rightarrow0$$
with $I$ injective. Then there is a map $p:A\rightarrow M$ such that $\pi p=1$ since $I$ is injective. It follows that $\pi(\phi-p\varphi)=0$. By the Universal Property of the kernel, we get a map $\alpha:T(B)\rightarrow I$ such that $\phi-p\varphi=i\alpha$. Since $\varphi$ is monic (by Proposition \ref{thm:main1}) and $I$ is injective, there is a map $\beta:A\rightarrow I$ such that $\alpha=\beta\varphi$. Set $q=p+i\beta$, we have $$q\varphi=(p+i\beta)\varphi=p\varphi+i\beta\varphi=p\varphi+ i\alpha=\phi,$$
$$\ \pi q= \pi p+ \pi i p = \pi p =1.$$
Thus $\left(\begin{smallmatrix}\pi\\1\end{smallmatrix}\right)\left(\begin{smallmatrix}q\\1\end{smallmatrix}\right)=1$, and hence $\xi$ is split.
\epf

\begin{prop}\label{thm:main4}
 If $\mathscr{A}$ has enough injective objects, then $\langle{\bf p}({^{\bot}\mathcal {X}},{^{\bot}\mathcal {Y}})\rangle={^{\bot}\left(\begin{smallmatrix}\mathcal {X}\\\mathcal{Y} \end{smallmatrix}\right)}\bigcap{^{\bot}\left(\begin{smallmatrix}\mathcal {I}\\{0} \end{smallmatrix}\right)}$, where $\mathcal{I}$ is the class of injective objects in $\mathscr{A}$.
\end{prop}

\begin{proof}
Assume that $\mathscr{A}$ has enough injective objects and $\mathcal{I}$ is the class of injective objects in $\mathscr{A}$.
To prove that ${^{\bot}\left(\begin{smallmatrix}\mathcal {X}\\\mathcal{Y} \end{smallmatrix}\right)}\bigcap{^{\bot}\left(\begin{smallmatrix}\mathcal {I}\\0 \end{smallmatrix}\right)}\subseteq\langle{\bf p}({^{\bot}\mathcal {X}},{^{\bot}\mathcal {Y}})\rangle$,
we suppose that $\left(\begin{smallmatrix}A\\B\end{smallmatrix}\right)_\varphi$ is an object in ${^{\bot}\left(\begin{smallmatrix}\mathcal {X}\\\mathcal{Y} \end{smallmatrix}\right)}\bigcap{^{\bot}\left(\begin{smallmatrix}\mathcal {I}\\0 \end{smallmatrix}\right)}$. It follows from Proposition \ref{thm:main1} and Lemma \ref{orth5} that $\varphi$ is monic. Note that both ${^{\bot}\mathcal {X}}$ and ${^{\bot}\mathcal {Y}}$ are closed under extensions, and so it is sufficient to show $B\in{^{\bot}\mathcal {Y}}$ and coker$\varphi\in{^{\bot}\mathcal {X}}$ by Proposition \ref{thm:main1}.

Let $\zeta:0\rightarrow Y\rightarrow N\rightarrow B\rightarrow0$ be an exact sequence in $\mathscr{B}$ with $Y\in\mathcal {Y}$. Then there is an exact sequence $\xi:0\rightarrow\left(\begin{smallmatrix}0\\Y\end{smallmatrix}\right){\rightarrow} \left(\begin{smallmatrix}A\\N\end{smallmatrix}\right)\rightarrow\left(\begin{smallmatrix}A\\B\end{smallmatrix}\right)\rightarrow0$ in $(T\downarrow\mathscr{A})$. By hypothesis, the sequence $\xi$ is split since $\left(\begin{smallmatrix}0\\Y\end{smallmatrix}\right)\in\left(\begin{smallmatrix}\mathcal {X}\\\mathcal{Y} \end{smallmatrix}\right)$. So the sequence $\zeta$ is split and $B\in{^{\bot}\mathcal {Y}}$, as desired.

Let $0\rightarrow X\stackrel{f}{\rightarrow} M{\rightarrow} \mbox{coker}\varphi\rightarrow0$ be an exact sequence in $\mathscr{A}$ with $X\in\mathcal{X} $. Then we have a pullback diagram $$\xymatrix{&&0\ar[d]&0\ar[d]\\&&T(B)\ar@{=}[r]\ar[d]^{\tilde{\varphi}}&T(B)\ar[d]^\varphi\\
0\ar[r]&X\ar[r]^{\tilde{f}}\ar@{=}[d]&L\ar[r]^\pi\ar[d]^p&A\ar[r]\ar[d]&0\\
    0\ar[r]&X \ar[r]^f&M\ar[d]\ar[r]&\mbox{coker}\varphi\ar[d] \ar[r]&0\\
    &&0&0,\\}$$ which induces an exact sequence $\xi':0\rightarrow\left(\begin{smallmatrix}X\\0\end{smallmatrix}\right)\stackrel{\left(\begin{smallmatrix}\tilde{f}\\0\end{smallmatrix}\right)}
{\rightarrow}\left(\begin{smallmatrix}L\\B\end{smallmatrix}\right)_{\widetilde{\varphi}}
\stackrel{\left(\begin{smallmatrix}\pi\\1\end{smallmatrix}\right)}
{\rightarrow}\left(\begin{smallmatrix}A\\B\end{smallmatrix}\right)_{\varphi}\rightarrow0$ in $(T\downarrow\mathscr{A})$. Since $\xi'$ is split by hypothesis, there is $\tilde{g}:L\rightarrow X$ such that $\tilde{g}\tilde{f}=1$ and $\tilde{g}\tilde{\varphi}=0$. By the Universal Property of the cokernel, there is a morphism ${g}:M\rightarrow X$ such that $\tilde{g}=g p$. Thus we have $gf=gp\tilde{f}=\tilde{g}\tilde{f}=1$ which means the third row in the above diagram is split. So  coker$\varphi\in{^{\bot}\mathcal {X}}$, as required.

For the reverse containment $\langle{\bf p}({^{\bot}\mathcal {X}},{^{\bot}\mathcal {Y}})\rangle \subseteq{^{\bot}\left(\begin{smallmatrix}\mathcal {X}\\\mathcal{Y} \end{smallmatrix}\right)}\bigcap{^{\bot}\left(\begin{smallmatrix}\mathcal {I}\\0 \end{smallmatrix}\right)}$, by Remark \ref{functors}(1) and Lemma \ref{orth5}, we only need to show that ${\bf p}(A,0), {\bf p}(0,B)\in{^{\bot}\left(\begin{smallmatrix}\mathcal {X}\\\mathcal{Y} \end{smallmatrix}\right)}$ for any $A\in{^{\bot}\mathcal {X}}$ and $B\in{^{\bot}\mathcal {Y}}$.

Let $\varepsilon:0\rightarrow\left(\begin{smallmatrix}X\\Y \end{smallmatrix}\right){\rightarrow} \left(\begin{smallmatrix}M\\Y\end{smallmatrix}\right){\rightarrow}\left(\begin{smallmatrix}A\\0\end{smallmatrix}\right)\rightarrow0$ be an exact sequence in $(T\downarrow\mathscr{A})$ with $A\in{^{\bot}\mathcal {X}}$, $X\in{\mathcal{X}}$ and $Y\in{\mathcal{Y}}$. Note that the exact sequence $0\rightarrow X\rightarrow M\rightarrow A\rightarrow0$ in $\mathscr{A}$ is split. It follows that the sequence $\varepsilon$ is split. So ${\bf p}(A,0)\in{^{\bot}\left(\begin{smallmatrix}\mathcal {X}\\\mathcal{Y} \end{smallmatrix}\right)}$. Similarly, one can prove that ${\bf p}(0,B)\in{^{\bot}\left(\begin{smallmatrix}\mathcal {X}\\\mathcal{Y} \end{smallmatrix}\right)}$. This completes the proof.
\end{proof}

Recall that an exact category $\mathscr{D}$ is said to be \emph{Frobenius} provided that it has enough projective and
enough injective objects, and the class of projective objects coincides with the class of injective
objects. We end this section with the following result which is a generalization of \cite[Corollary 2.3]{XZZ}.

\begin{cor}\label{Frobenius} Let $\mathscr{A}$ and $\mathscr{B}$ be abelian categories with enough projective objects and enough injective objects. If $T:\mathscr{B}\rightarrow\mathscr{A}$ is an exact functor, then $\langle{\bf p}(\mathscr{A},\mathscr{B})\rangle$ is a Frobenius category if and only if $\mathscr{A}$ and $\mathscr{B}$ are Frobenius and $T$ preserves projective objects.
\end{cor}

\pf Clearly, $\langle{\bf p}(\mathscr{A},\mathscr{B})\rangle$ is an exact category where the exact structure inherits from $(T\downarrow\mathscr{A})$.
Let $\mathcal{P}$ (\textrm{resp.} $\mathcal{I}$) be the class of projective (resp. injective) objects in $\mathscr{A}$ and $\mathcal{Q}$ (\textrm{resp.} $\mathcal{J}$) the class of projective (resp. injective) objects in $\mathscr{B}$.

Then we have
$$\begin{array}{lll}{^{\bot}{\bf p}(\mathscr {A},0)}\bigcap{^{\bot}\left(\begin{smallmatrix}\mathcal {I}\\0 \end{smallmatrix}\right)}&={^{\bot}\left(\begin{smallmatrix}\mathscr {A}\\0 \end{smallmatrix}\right)}\bigcap{^{\bot}\left(\begin{smallmatrix}\mathcal {I}\\0 \end{smallmatrix}\right)}&(\mbox{since }{\bf p}(\mathscr {A},0)=\left(\begin{smallmatrix}\mathscr {A}\\0 \end{smallmatrix}\right) )\\&=\langle{\bf p}({^{\bot}\mathscr{A}},{^{\bot}0})\rangle&(\mbox{by Proposition \ref{thm:main4}})\\&=\langle{\bf p}(\mathcal{P},\mathscr{B})\rangle,\end{array}$$

$$\begin{array}{lll}{^{\bot}{\bf p}(0,\mathscr {B})}\bigcap{^{\bot}\left(\begin{smallmatrix}\mathcal {I}\\0 \end{smallmatrix}\right)}&={^{\bot}\left(\begin{smallmatrix}T(\mathscr {B})\\\mathscr {B} \end{smallmatrix}\right)}\bigcap{^{\bot}\left(\begin{smallmatrix}\mathcal {I}\\0 \end{smallmatrix}\right)}&(\mbox{since }{\bf p}(0,\mathscr {B})=\left(\begin{smallmatrix}T(\mathscr {B})\\\mathscr {B} \end{smallmatrix}\right) )\\&=\langle{\bf p}({^{\bot}T(\mathscr{B})},{^{\bot}\mathscr{B}})\rangle&(\mbox{by Proposition \ref{thm:main4}})\\&=\langle{\bf p}({^{\bot}T(\mathscr{B})},\mathcal{Q})\rangle.\end{array}$$

Thus, we get
$$\begin{array}{lll}{^{\bot}\langle{\bf p}(\mathscr{A},\mathscr{B})\rangle}&={^{\bot}{\bf p}(\mathscr{A},\mathscr{B})}={^{\bot}{\bf p}(\mathscr {A},0)}\bigcap{^{\bot}{\bf p}(0,\mathscr {B})}&(\mbox{by Remarks \ref{functors}(1)})\\&={^{\bot}{\bf p}(\mathscr {A},0)}\bigcap{^{\bot}\left(\begin{smallmatrix}\mathcal {I}\\0 \end{smallmatrix}\right)}\bigcap{^{\bot}{\bf p}(0,\mathscr {B})}&(\mbox{since }\mathcal{I}\subset\mathscr{A})\\&=\left({^{\bot}{\bf p}(\mathscr {A},0)}\bigcap{^{\bot}\left(\begin{smallmatrix}\mathcal {I}\\0 \end{smallmatrix}\right)}\right)\bigcap\left({^{\bot}{\bf p}(0,\mathscr {B})}\bigcap{^{\bot}\left(\begin{smallmatrix}\mathcal {I}\\0 \end{smallmatrix}\right)}\right)\\&=\langle{\bf p}(\mathcal{P},\mathscr{B})\rangle\bigcap\langle{\bf p}({^{\bot}T(\mathscr{B})},\mathcal{Q})\rangle&(\mbox{by above identities})\\&=\langle{\bf p}(\mathcal{P},\mathcal{Q})\rangle.\end{array}$$

Moreover, the class of injective objects in the exact category $\langle{\bf p}(\mathscr{A},\mathscr{B})\rangle$ is $\left(\begin{smallmatrix}\mathcal {I}\\\mathcal{J} \end{smallmatrix}\right)\bigcap\langle{\bf p}(\mathscr{A},\mathscr{B})\rangle$.
Therefore, $\langle{\bf p}(\mathscr{A},\mathscr{B})\rangle$ is Frobenius if and only if $T(\mathcal{Q})\subset\mathcal{I}=\mathcal{P}$ and $\mathcal{J}=\mathcal{Q}$, as desired.
\epf

\section{Complete cotorsion pairs and the proof of Theorem \ref{thm:special precovering}(1)}
In this section, we shall use our results in Section 2 to show the first statement of the main
result, Theorem \ref{thm:special precovering}. More precisely, we first recall the definition of complete hereditary cotorsion pairs in abelian categories, and then characterize when complete hereditary  cotorsion pairs in abelian categories $\mathscr{A}$ and $\mathscr{B}$ can induce complete hereditary cotorsion pairs in $(T\downarrow\mathscr{A})$. Especially, we shall establish
a crucial result, Proposition \ref{thm:cotorsion-pair1}, which will play a role in the proof of Theorem \ref{thm:special precovering}(1).

Let $\mathcal{L}$ be a class of objects in an abelian category $\mathscr{D}$. Recall that $\mathcal{L}$ is said to  be \emph{resolving} if whenever $0\ra X\ra Y\ra Z\ra0$  is exact in $\mathscr{D}$ with $Z\in{\mathcal{L}}$, then $X\in{\mathcal{L}}$ if and only if $Y\in{\mathcal{L}}$. Dually, $\mathcal{L}$ is said to  be \emph{coresolving} if whenever $0\ra X\ra Y\ra Z\ra0$  is exact in $\mathscr{D}$ with $X\in{\mathcal{L}}$, then $Y\in{\mathcal{L}}$ if and only if $Z\in{\mathcal{L}}$.

\begin{df}
Let $\mathscr{D}$ be an abelian category.
 \begin{enumerate}
\item  A \emph{cotorsion pair} \cite{SL} is a pair of classes ($\mathcal{H}$, $\mathcal{G}$) of objects in $\mathscr{D}$ such that $\mathcal{H}^{{\perp}}=\mathcal{G}$ and $^{{ \perp}}\mathcal{G}=\mathcal{H}$.
\item A cotorsion pair ($\mathcal{H}$, $\mathcal{G}$) is said to be \emph{hereditary} \cite{HC} if ${\rm Ext}^i_\mathscr{D}(X,Y)=0$ for every $X\in{\mathcal{H}}$ and $Y\in{\mathcal{G}}$, and $i>0$. Moreover, if $\mathscr{D}$ has enough projectives and injectives, the condition that ($\mathcal{H}$, $\mathcal{G}$) is hereditary is equivalent to that $\mathcal{H}$ is resolving and $\mathcal{G}$ is coresolving (see \cite[Proposition 2.1]{EPZ}).

\item A cotorsion pair ($\mathcal{H}$, $\mathcal{G}$) is said to be \emph{complete} \cite{Gober} if $\mathcal{H}$ is special precovering and $\mathcal{G}$ is special preenveloping. Moreover, if $\mathscr{D}$ has enough projective objects, the condition that ($\mathcal{H}$, $\mathcal{G}$) is complete is equivalent to that $\mathcal{G}$ is special preenveloping. Similarly, if $\mathscr{D}$ has enough injective objects, the condition that ($\mathcal{H}$, $\mathcal{G}$) is complete is equivalent to that $\mathcal{H}$ is special precovering (see \cite[6.3, p.595]{HJ}).
\end{enumerate}
\end{df}

\begin{lem}\label{resov}  The following are true for any comma category $(T\downarrow\mathscr{A})$:

 \begin{enumerate}
\item  $\left(\begin{smallmatrix}\mathcal {X}\\\mathcal{Y} \end{smallmatrix}\right)$ is coresolving in $(T\downarrow\mathscr{A})$ if and only if $\mathcal {X}$ and $\mathcal {Y}$ are coresolving in $\mathscr{A}$ and $\mathscr{B}$ respectively.
\item  If $T:\mathscr{B}\rightarrow\mathscr{A}$ is $\mathcal {Y}$-exact and $\mathcal {X}$, $\mathcal {Y}$ are closed under extensions in $\mathscr{A}$ and $\mathscr{B}$, respectively, then $\langle{\bf p}(\mathcal {X},\mathcal {Y})\rangle$ is resolving in $(T\downarrow\mathscr{A})$ if and only if $\mathcal {X}$ and $\mathcal {Y}$ are resolving in $\mathscr{A}$ and $\mathscr{B}$, respectively.

   \end{enumerate}
\end{lem}

\pf  We only prove (2); the proof of (1) is straightforward. Assume that $T$ is $\mathcal {Y}$-exact. For the ``only if" part, we assume that $\langle{\bf p}(\mathcal {X},\mathcal {Y})\rangle$ is resolving in $(T\downarrow\mathscr{A})$. Let $0\rightarrow X'\rightarrow X\rightarrow X''\rightarrow0$ be an exact sequence in $\mathscr{A}$ with $X''$ $\in$ $\mathcal{X}$. It follows that the sequence $0\rightarrow{\bf p}(X',0)\rightarrow{\bf p}(X,0)\rightarrow{\bf p}(X'',0)\rightarrow0 $ is exact in $(T\downarrow\mathscr{A})$. Thus we get that ${\bf p}(X,0)\in\langle{\bf p}(\mathcal {X},\mathcal {Y})\rangle$ if and only if  ${\bf p}(X',0)\in\langle{\bf p}(\mathcal {X},\mathcal {Y})\rangle$ by hypothesis, and hence $X\in{\mathcal{X}}$ if and only if $X'\in{\mathcal{X}}$. So $\mathcal{X}$ is resolving in $\mathscr{A}$. Similarly, one can prove that $\mathcal{Y}$ is resolving in $\mathscr{B}$.

For the ``if" part, we assume that  $\mathcal {X}$ and $\mathcal {Y}$ are resolving in $\mathscr{A}$ and $\mathscr{B}$, respectively. Then $\mathcal {X}$ and $\mathcal {Y}$ are closed under extensions in $\mathscr{A}$ and $\mathscr{B}$ respectively. Let  $0\rightarrow\left(\begin{smallmatrix}X'\\Y' \end{smallmatrix}\right)_{\varphi'}\rightarrow \left(\begin{smallmatrix}X\\Y\end{smallmatrix}\right)_{\varphi}\rightarrow\left(\begin{smallmatrix}X''\\Y''\end{smallmatrix}\right)_{\varphi''}\rightarrow0 $ be an exact sequence in $(T\downarrow\mathscr{A})$ with $\left(\begin{smallmatrix}X''\\Y'' \end{smallmatrix}\right)_{\varphi''}\in\langle{\bf p}(\mathcal {X},\mathcal {Y})\rangle$. Note that $T$ is $\mathcal {Y}$-exact. Then we have the following commutative diagram of exact sequences:

$$\xymatrix{0\ar[r]&T(Y')\ar[r]\ar[d]^{\varphi'}&T(Y)\ar[r]\ar[d]^\varphi&T(Y'')\ar[r]\ar[d]^{\varphi''}&0\\
    0\ar[r]&X' \ar[r]&X\ar[r]&X'' \ar[r]&0.}$$
Note that $\varphi''$ is monic and $\mbox{coker}\varphi''\in \mathcal{X}$ by Proposition \ref{thm:main1}. Then we have that ${\ker}\varphi'\cong{\ker}\varphi$ and
$0\rightarrow\mbox{coker}\varphi'\rightarrow\mbox{coker}\varphi\rightarrow\mbox{coker}\varphi''\rightarrow0$ is exact in $\mathscr{A}$. Therefore, we get that $\varphi$ is monic if and only if $\varphi'$ is monic. Since $\mathcal{X}$ is resolving in $\mathscr{A}$ by hypothesis, we have that $\mbox{coker}\varphi\in \mathcal {X}$ if and only if $\mbox{coker}\varphi'\in \mathcal {X}$. Note that $0\rightarrow Y' \rightarrow Y\rightarrow Y''\rightarrow0$ is an exact sequence in $\mathscr{B}$ and $\mathcal{Y}$ is resolving in $\mathscr{B}$. It follows that $Y'\in\mathcal{Y}$ if and only if $Y\in\mathcal{Y}$. So  $\left(\begin{smallmatrix}X\\Y \end{smallmatrix}\right)_{\varphi}\in\langle{\bf p}(\mathcal {X},\mathcal {Y})\rangle$ if and only if $\left(\begin{smallmatrix}X'\\Y' \end{smallmatrix}\right)_{\varphi'}\in\langle{\bf p}(\mathcal {X},\mathcal {Y})\rangle$ by Proposition \ref{thm:main1}. This completes the proof.
\epf

\begin{lem}\label{lem:cotorsion-pair} Let $\mathscr{A}$ and $\mathscr{B}$ both have enough projective objects and enough injective objects.
If $T:\mathscr{B}\rightarrow\mathscr{A}$ is $\mathcal {Y}$-exact, then $(\mathcal {X}, \mathcal {X}^{\bot})$ and $(\mathcal {Y}, \mathcal {Y}^{\bot})$ are (hereditary) cotorsion pairs in $\mathscr{A}$ and $\mathscr{B}$, respectively if and only if $(\langle{\bf p}(\mathcal {X},\mathcal {Y})\rangle, \left(\begin{smallmatrix}\mathcal {X}^{\bot}\\\mathcal {Y}^{\bot}\end{smallmatrix}\right))$ is a (hereditary) cotorsion pair in $(T\downarrow\mathscr{A})$ and $\mathcal{X}$, $\mathcal {Y}$ are closed under extensions.
\end{lem}
\begin{proof}
By \cite[Proposition 2.1]{EPZ} and Lemma \ref{resov}, it suffices to show that $(\langle{\bf p}(\mathcal {X},\mathcal {Y})\rangle, \left(\begin{smallmatrix}\mathcal {X}^{\bot}\\\mathcal {Y}^{\bot}\end{smallmatrix}\right))$ is a cotorsion pair in $(T\downarrow\mathscr{A})$ if and only if $(\mathcal {X}, \mathcal {X}^{\bot})$ and $(\mathcal {Y}, \mathcal {Y}^{\bot})$ are cotorsion pairs in $\mathscr{A}$ and $\mathscr{B}$, respectively.

``$\Rightarrow$". Let $(\mathcal {X}, \mathcal {X}^{\bot})$ be a cotorsion pair in $\mathscr{A}$  and $(\mathcal {Y}, \mathcal {Y}^{\bot})$ a cotorsion pair in $\mathscr{B}$. Then  all injective objects belong to $\mathcal {X}^{\bot}$. By Proposition \ref{thm:main4}, we have the following
$$\langle{\bf p}(\mathcal {X},\mathcal {Y})\rangle=\langle{\bf p}({^{\bot}(\mathcal {X}^{\bot})},{^{\bot}(\mathcal {Y}^{\bot}))}\rangle={^{\bot}\left(\begin{smallmatrix}\mathcal {X}^{\bot}\\\mathcal {Y}^{\bot}\end{smallmatrix}\right).}$$
\noindent Note that ${\langle{\bf p}(\mathcal {X},\mathcal {Y})\rangle}^{\bot}={\left(\begin{smallmatrix}\mathcal {X}^{\bot}\\\mathcal {Y}^{\bot}\end{smallmatrix}\right)}$ by Proposition \ref{thm:main3}. It follows that $(\langle{\bf p}(\mathcal {X},\mathcal {Y})\rangle, \left(\begin{smallmatrix}\mathcal {X}^{\bot}\\\mathcal {Y}^{\bot}\end{smallmatrix}\right))$ is a cotorsion pair in $(T\downarrow\mathscr{A})$.

``$\Leftarrow$". Now, we assume that $(\langle{\bf p}(\mathcal {X},\mathcal {Y})\rangle, \left(\begin{smallmatrix}\mathcal {X}^{\bot}\\\mathcal {Y}^{\bot}\end{smallmatrix}\right))$ is a cotorsion pair in $(T\downarrow\mathscr{A})$ and $\mathcal {X}$, $\mathcal {Y}$ are closed under extensions. It is sufficient to show that  ${^{\bot}(\mathcal {X}^{\bot})}\subseteq\mathcal {X}$ and ${^{\bot}(\mathcal {Y}^{\bot})}\subseteq\mathcal {Y}$. Let $A\in{^{\bot}(\mathcal {X}^{\bot})}$ and $B\in{^{\bot}(\mathcal {Y}^{\bot})}$. Then for any $M\in\mathcal {X}^{\bot}$ and $N\in\mathcal {Y}^{\bot}$, it is clear that the following exact sequences $$0\rightarrow  \left(\begin{smallmatrix}M\\N\end{smallmatrix}\right) \rightarrow \left(\begin{smallmatrix}C\\N\end{smallmatrix}\right)\rightarrow \left(\begin{smallmatrix}A\\0 \end{smallmatrix}\right)\rightarrow0 $$ and $$0\rightarrow \left(\begin{smallmatrix}M\\N\end{smallmatrix}\right) \rightarrow \left(\begin{smallmatrix}D\\L\end{smallmatrix}\right)\rightarrow \left(\begin{smallmatrix}T(B)\\B\end{smallmatrix}\right)\rightarrow0 $$
are split. This implies that $\left(\begin{smallmatrix}A\\0 \end{smallmatrix}\right), \left(\begin{smallmatrix}T(B)\\B\end{smallmatrix}\right)\in {^{\bot}\left(\begin{smallmatrix}\mathcal {X}^{\bot}\\\mathcal {Y}^{\bot}\end{smallmatrix}\right)}=\langle{\bf p}(\mathcal {X},\mathcal {Y})\rangle$, and so $A\in\mathcal{X}, B\in\mathcal{Y}$ by Proposition \ref{thm:main1}. This completes the proof.
\end{proof}

The following proposition is crucial to the proof of Theorem \ref{thm:special precovering}(1).

\begin{prop}\label{thm:cotorsion-pair1}Let $\mathscr{A}$ and $\mathscr{B}$ both have enough projective objects and enough injective objects. Assume that $T:\mathscr{B}\rightarrow\mathscr{A}$ is $\mathcal {Y}$-exact. If $(\mathcal {X}, \mathcal {X}^{\bot})$ and $(\mathcal {Y}, \mathcal {Y}^{\bot})$ are complete cotorsion pairs in $\mathscr{A}$ and $\mathscr{B}$, respectively, then so is $(\langle{\bf p}(\mathcal {X},\mathcal {Y})\rangle, \left(\begin{smallmatrix}\mathcal {X}^{\bot}\\\mathcal {Y}^{\bot}\end{smallmatrix}\right))$. Moreover, the converse holds when $T(\mathcal {Y} \bigcap \mathcal {Y}^{\bot})\subseteq{\mathcal {X}^{\bot}}$ and $\mathcal X, \mathcal Y$ are closed under extensions.
\end{prop}

\begin{proof}
Assume that $(\mathcal {X}, \mathcal {X}^{\bot})$ is a complete cotorsion pair in $\mathscr{A}$ and $(\mathcal {Y}, \mathcal {Y}^{\bot})$ is a complete cotorsion pair in $\mathscr{B}$. For any $\left(\begin{smallmatrix}A\\B\end{smallmatrix}\right)\in$ $(T\downarrow\mathscr{A})$, there is an exact sequence $0\rightarrow B\rightarrow V\rightarrow Y\rightarrow0$ in $\mathscr{B}$ with $V\in\mathcal {Y}^{\bot}$ and $Y\in \mathcal {Y}$. Note that $T$ is $\mathcal {Y}$-exact, we get the following pushout diagram $$\xymatrix{
0\ar[r]&T(B)\ar[r]\ar[d]&T(V)\ar[r]\ar[d]&T(Y)\ar[r]\ar@{=}[d]&0\\
    0\ar[r]&A \ar[r]&C\ar[r]&T(Y) \ar[r]&0.\\
    }$$
Furthermore, we have an exact sequence $0\rightarrow C\rightarrow U\rightarrow X\rightarrow0$ in $\mathscr{A}$ with $U\in\mathcal {X}^{\bot}$ and $X\in \mathcal {X}$. Consider the following pushout diagram
$$\xymatrix{&&0\ar[d]&0\ar[d]\\0\ar[r]&A\ar[r]\ar@{=}[d]&C\ar[r]\ar[d]&T(Y)\ar[r]\ar[d]^\varphi&0\\
    0\ar[r]&A\ar[r]&U\ar[d]\ar[r]&D\ar[d] \ar[r]&0\\
    &&X\ar@{=}[r]\ar[d]&X\ar[d]\\&&0&0.\\}$$
Thus we get an exact sequence $0\rightarrow\left(\begin{smallmatrix}A\\B\end{smallmatrix}\right){\rightarrow} \left(\begin{smallmatrix}U\\V\end{smallmatrix}\right)\rightarrow\left(\begin{smallmatrix}D\\Y\end{smallmatrix}\right)_\varphi\rightarrow0$ in $(T\downarrow\mathscr{A})$ with $\left(\begin{smallmatrix}U\\V\end{smallmatrix}\right)\in\left(\begin{smallmatrix}\mathcal {X}^{\bot}\\\mathcal {Y}^{\bot}\end{smallmatrix}\right)$ and $\left(\begin{smallmatrix}D\\Y\end{smallmatrix}\right)_\varphi\in\langle{\bf p}(\mathcal {X},\mathcal {Y})\rangle$. Note that $\mathscr{A}$ and $\mathscr{B}$ have enough projective objects. Then $(T\downarrow\mathscr{A})$ has enough projective objects. So $(\langle{\bf p}(\mathcal {X},\mathcal {Y})\rangle, \left(\begin{smallmatrix}\mathcal {X}^{\bot}\\\mathcal {Y}^{\bot}\end{smallmatrix}\right))$ is a complete cotorsion pair in $(T\downarrow\mathscr{A})$ by \cite[6.3, p.595]{HJ}, as desired.

Conversely, we assume that $(\langle{\bf p}(\mathcal {X},\mathcal {Y})\rangle, \left(\begin{smallmatrix}\mathcal {X}^{\bot}\\\mathcal {Y}^{\bot}\end{smallmatrix}\right))$ is a complete cotorsion pair in $(T\downarrow\mathscr{A})$, $T(\mathcal {Y} \bigcap \mathcal {Y}^{\bot})\subseteq{\mathcal {X}^{\bot}}$ and $\mathcal X, \mathcal Y$ are closed under extensions. Then for any $B\in\mathscr{B}$, we have an exact sequence  $0\rightarrow\left(\begin{smallmatrix}U\\V\end{smallmatrix}\right){\rightarrow} \left(\begin{smallmatrix}U\\Y\end{smallmatrix}\right)\rightarrow\left(\begin{smallmatrix}0\\B\end{smallmatrix}\right)\rightarrow0$ in $(T\downarrow\mathscr{A})$ with $\left(\begin{smallmatrix}U\\Y\end{smallmatrix}\right)\in\langle{\bf p}(\mathcal {X},\mathcal {Y})\rangle$ and $\left(\begin{smallmatrix}U\\V\end{smallmatrix}\right)\in\left(\begin{smallmatrix}\mathcal {X}^{\bot}\\\mathcal {Y}^{\bot}\end{smallmatrix}\right)$. Thus we have an exact sequence $0\rightarrow V\rightarrow Y\rightarrow B \rightarrow0$ in ${\mathscr{B}}$ with $Y\in\mathcal{Y}$ and $V\in\mathcal {Y}^{\bot}$ which means that $(\mathcal {Y}, \mathcal {Y}^{\bot})$ is a complete cotorsion pair in $\mathscr{B}$.

In further, for any $A\in\mathscr{A}$, we have an exact sequence  $0\rightarrow\left(\begin{smallmatrix}K\\Y\end{smallmatrix}\right)_{\phi}{\rightarrow} \left(\begin{smallmatrix}M\\Y\end{smallmatrix}\right)_\varphi\rightarrow\left(\begin{smallmatrix}A\\0\end{smallmatrix}\right)\rightarrow0$ in $(T\downarrow\mathscr{A})$  with $\left(\begin{smallmatrix}M\\Y\end{smallmatrix}\right)_\varphi\in\langle{\bf p}(\mathcal {X},\mathcal {Y})\rangle$ and $\left(\begin{smallmatrix}K\\Y\end{smallmatrix}\right)_\phi\in\left(\begin{smallmatrix}\mathcal {X}^{\bot}\\\mathcal {Y}^{\bot}\end{smallmatrix}\right)$. Thus we obtain that $Y\in\mathcal {Y}\bigcap\mathcal {Y}^{\bot}, \coker\varphi\in\mathcal{X}, K\in\mathcal {X}^{\bot}$ and $\varphi$ is monic. Consequently,
we have the following exact commutative diagram:
$$\xymatrix{&0\ar[d]&0\ar[d]\\&T(Y)\ar@{=}[r]\ar[d]^{\phi}&T(Y)\ar[d]^\varphi\\
0\ar[r]&K\ar[r]^f\ar[d]^{\tilde{\pi}}&M\ar[r]\ar[d]^\pi&A\ar[r]\ar@{=}[d]&0\\
    0\ar[r]&\coker\phi \ar[d]\ar[r]^{\tilde{f}}&\coker\varphi\ar[d]\ar[r]&A \ar[r]&0\\
    &0&0.\\}$$
By hypotheses, $T(Y)\in\mathcal{X}^{\bot}$. It follows that the middle column splits which implies that the left column is split, and hence
 $\coker\phi$ is a direct summand of $K$. Thus $\coker\phi\in\mathcal {X}^{\bot}$, that is, the exact sequence $0\rightarrow\coker\phi\rightarrow\coker\varphi\rightarrow A\rightarrow0$ implies $(\mathcal {X}, \mathcal {X}^{\bot})$ is a complete cotorsion pair in $\mathscr{A}$.
\end{proof}
\begin{rem}\label{remark:3.5} {\rm If we replace the condition ``$T(\mathcal {Y} \bigcap \mathcal {Y}^{\bot})\subseteq{\mathcal {X}^{\bot}}$" with ``$T(\mathcal {Y} \bigcap \mathcal {Y}^{\bot})\subseteq{\mathcal {X}}$" in Proposition \ref{thm:cotorsion-pair1}, the result still holds. In the case $T(\mathcal {Y} \bigcap \mathcal {Y}^{\bot})\subseteq{\mathcal {X}}$, for any $A\in\mathscr{A}$, it is easy to check that the exact sequence $0\to K\to M\to A\to$ in the third commutative diagram in Proposition \ref{thm:cotorsion-pair1} satisfies that $K\in{\mathcal{X}^{\perp}}$ and $M\in{\mathcal{X}}$. This implies that $(\mathcal {X}, \mathcal {X}^{\bot})$ is a complete cotorsion pair in $\mathscr{A}$.}
\end{rem}

As a corollary of Proposition \ref{thm:cotorsion-pair1} and Corollary \ref{corollary:3.6}, we re-obtain \cite[Theorem 1.1(2)]{XZZ}.
\begin{cor} (see \cite{XZZ})\label{corollary:4.6} Let $R$ and $S$ be Artin algebras, and $\Lambda=\left(
                                                                                 \begin{smallmatrix}
                                                                                   R & M \\
                                                                                   0 & S \\
                                                                                 \end{smallmatrix}
                                                                               \right)
$ a triangular matrix algebra. If $M$ is finitely generated projective left $S$-module, then $\mathcal{M}(R, M, S)$ is a functorially finite subcategory of ${\rm mod}\Lambda$, and has Auslander-Reiten sequences.
\end{cor}
\begin{proof} By the proof of Theorem 1.1(2) in \cite[p.34]{XZZ}, it suffices to show that $\mathcal{M}(R, M, S)$ is precovering in ${\rm mod}\Lambda$. This is true by Proposition \ref{thm:cotorsion-pair1} and Corollary \ref{corollary:3.6}.
\end{proof}

\begin{cor}\label{corollary:4.5'} Let $\Lambda=\left(
                                                                                 \begin{smallmatrix}
                                                                                   R & M \\
                                                                                   0 & S \\
                                                                                 \end{smallmatrix}
                                                                               \right)
$ be a triangular matrix ring, and let $\mathcal{X}_{1}$ and $\mathcal{X}_{2}$ be two classes of left $R$-modules, $\mathcal{Y}_1$ and $\mathcal{Y}_{2}$ be two classes of left $S$-modules. Assume that ${\rm Tor}_{1}^{S}(M,Y)=0$ for any $Y\in{\mathcal{Y}_1}$.
If $(\mathcal{X}_1, \mathcal{X}_2)$ and $(\mathcal{Y}_1, \mathcal{Y}_2)$ are hereditary complete cotorsion pairs in $\Mod R$ and $\Mod S$, respectively, then so is $(\mathfrak{B}^{\mathcal{X}_1}_{\mathcal{Y}_1}, \left(\begin{smallmatrix}\mathcal{X}_2\\ \mathcal{Y}_2\end{smallmatrix}\right))$. Moreover, the converse holds when $M\otimes_{S}N\subseteq{\mathcal {X}_2}$ or $M\otimes_{S}N\subseteq{\mathcal{X}_1}$ for any $N \in {\mathcal {Y}_1 \bigcap \mathcal {Y}_2}$.
\end{cor}
\begin{proof} If we define $T\cong M\otimes_{S}-: \Mod S \ra \Mod R$, then $\Mod {\Lambda}$ is equivalent to the comma category $(T\downarrow \Mod R)$ by Example \ref{e exact}(1). Assume that $(\mathfrak{B}^{\mathcal{X}_1}_{\mathcal{Y}_1}, \left(\begin{smallmatrix}\mathcal{X}_2\\ \mathcal{Y}_2\end{smallmatrix}\right))$ is a cotorsion pair. Then $0\in{\mathcal{X}_1\cap {\mathcal{X}_2}}$, $0\in{\mathcal{Y}_1\cap {\mathcal{Y}_2}}$ and $\mathfrak{B}^{\mathcal{X}_1}_{\mathcal{Y}_1}$ is closed under extensions. It follows from Proposition \ref{thm:main1} that $\langle{\bf p}(\mathcal {X}_1,\mathcal {Y}_1)\rangle=\mathfrak{B}^{\mathcal{X}_1}_{\mathcal{Y}_1}$. Let $M$ be a left $R$-module in $\mathcal{X}_2$. Then $\left(\begin{smallmatrix}M\\0\end{smallmatrix}\right)\in \left(\begin{smallmatrix}\mathcal{X}_2\\\mathcal{Y}_2\end{smallmatrix}\right)=\langle{\bf p}(\mathcal {X}_1,\mathcal {Y}_1)\rangle^{\perp}=\left(\begin{smallmatrix}\mathcal{X}_{1}^{\bot}\\\mathcal {Y}_1^{\bot}\end{smallmatrix}\right)$ by Lemma \ref{thm:main3}. This implies that $M\in{\mathcal{X}_{1}^{\bot}}$. On the other hand, let $N$ be a left $R$-module in $\mathcal{X}_{1}^{\bot}$. It follows that $\left(\begin{smallmatrix}N\\0\end{smallmatrix}\right)\in \left(\begin{smallmatrix}\mathcal{X}_1^{\perp}\\\mathcal{Y}_1^{\perp}\end{smallmatrix}\right)=\langle{\bf p}(\mathcal {X}_1,\mathcal {Y}_1)\rangle^{\perp}=\left(\begin{smallmatrix}\mathcal{X}_{2}\\\mathcal {Y}_2\end{smallmatrix}\right)$ by Lemma \ref{thm:main3}. So we have $N\in{\mathcal{X}_2}$ and $\mathcal{X}_1^{\perp}=\mathcal{X}_2$. Similarly, one can show $\mathcal{Y}_1^{\perp}=\mathcal{Y}_2$. So the result holds by Lemma \ref{lem:cotorsion-pair}, Proposition \ref{thm:cotorsion-pair1} and Remark \ref{remark:3.5}.
\end{proof}

\begin{rem}\label{rem:3.7} {\rm We note that Corollary \ref{corollary:4.5'} refines a result obtained by Mao in \cite{Mao}. More precisely, the conditions that ``${\rm Tor}_{i}^{A}(U,^{\perp}\mathcal{C}_{2})=0$ for any $i\geq1$" and ``${\rm Tor}_{i}^{A}(U,\mathcal{C}_{1})=0$ for any $i\geq2$" in \cite[Theorem 5.6(1)]{Mao} are superfluous. Also, our proof here is different from
that in \cite{Mao}.}
\end{rem}

Let $\mathcal{L}$ be a class of objects in an abelian category $\mathscr{D}$. We denote by $\mathsf{Smd}(\mathcal{L})$ the class of direct summands of objects in $\mathcal L$.

\begin{lem}\label{smd}Let $\mathscr{D}$ be an abelian category with enough injective objects. If $\mathcal L$ is special precovering in $\mathscr{D}$, then $(\mathsf{Smd}(\mathcal{L}),\mathsf{Smd}(\mathcal{L})^\bot)$ is a complete cotorsion pair.
\end{lem}

\pf Note that $\mathsf{Smd}(\mathcal{L})^\bot=\mathcal{L}^\bot$. Then it is clear that $\mathsf{Smd}(\mathcal{L})$ is a special precovering class. In the sequel, we claim that ${^\bot(\mathsf{Smd}(\mathcal{L})^\bot)}\subseteq\mathsf{Smd}(\mathcal{L})$. For any $L\in{^\bot(\mathsf{Smd}(\mathcal{L})^\bot)}$, there is an exact sequence $0\rightarrow K\rightarrow \overline{L}\rightarrow L\rightarrow0$ with $\overline{L}\in\mathcal L$ and $K\in\mathcal{L}^\bot$ since $\mathcal L$ is  special precovering. Thus the above sequence is split as $\mathsf{Smd}(\mathcal{L})^\bot=\mathcal{L}^\bot$. It follows that $L\in\mathsf{Smd}(\mathcal{L})$, and so ${^\bot(\mathsf{Smd}(\mathcal{L})^\bot)}=\mathsf{Smd}(\mathcal{L})$. We complete the proof.
\epf

We are now in a position to prove Theorem \ref{thm:special precovering}(1).

{\bf Proof of Theorem \ref{thm:special precovering}(1).} At first, we claim that $T$ is $\mathsf{Smd}(\mathcal {Y})$-exact. For any exact sequence $0\rightarrow B_1\rightarrow B_0\rightarrow {Y}\rightarrow0$ in $\mathscr{B}$ with ${Y}\in\mathsf{Smd}(\mathcal {Y})$, there is an induced exact sequence
$0\rightarrow B_1\rightarrow B_0\bigoplus M\rightarrow {Y}\bigoplus M\rightarrow0$ where ${ Y}\bigoplus M\in\mathcal {Y}$. Then we obtain the following commutative diagram $$\xymatrix{
0\ar[r]&T(B_1)\ar[r]\ar@{=}[d]&T(B_0)\ar[r]\ar[d]&T(Y)\ar[r]\ar[d]&0\\
    0\ar[r]&T(B_1)\ar[r]& T(B_0)\bigoplus T(M)\ar[r]&T(Y)\bigoplus T(M) \ar[r]&0,\\
    }$$
where the bottom row is exact since $T$ is $\mathcal {Y}$-exact. It follows that the first row is exact.

Since $(\mathsf{Smd}(\mathcal {X}), \mathsf{Smd}(\mathcal {X})^\bot)$ and $(\mathsf{Smd}(\mathcal {Y}), \mathsf{Smd}(\mathcal {Y})^\bot)$ are complete cotorsion pairs by Lemma \ref{smd}, the pair $(\langle{\bf p}(\mathsf{Smd}(\mathcal {X}),\mathsf{Smd}(\mathcal {Y}))\rangle, \left(\begin{smallmatrix}\mathsf{Smd}(\mathcal {X})^{\bot}\\\mathsf{Smd}(\mathcal {Y})^{\bot}\end{smallmatrix}\right))$ is a complete cotorsion pair in $(T\downarrow\mathscr{A})$ by Proposition \ref{thm:cotorsion-pair1}. Thus for any object $\left(\begin{smallmatrix}A\\B\end{smallmatrix}\right)\in(T\downarrow\mathscr{A})$, there is an exact sequence $0\rightarrow\left(\begin{smallmatrix}\overline{C}\\\overline{D}\end{smallmatrix}\right)_{\overline{\psi}}{\rightarrow} \left(\begin{smallmatrix}\overline{M}\\\overline{Y}\end{smallmatrix}\right)_{\overline{\varphi}}\rightarrow\left(\begin{smallmatrix}A\\B\end{smallmatrix}\right)\rightarrow0$ with $\left(\begin{smallmatrix}\overline{M}\\\overline{Y}\end{smallmatrix}\right)_{\overline{\varphi}}\in\langle{\bf p}(\mathsf{Smd}(\mathcal {X}),\mathsf{Smd}(\mathcal {Y}))\rangle$ and $\left(\begin{smallmatrix}\overline{C}\\\overline{D}\end{smallmatrix}\right)_{\overline{\psi}}\in\left(\begin{smallmatrix}\mathsf{Smd}(\mathcal {X})^{\bot}\\\mathsf{Smd}(\mathcal {Y})^{\bot}\end{smallmatrix}\right)$.

Since $\mathcal{Y}$ is special precovering, we have an exact sequence $0\rightarrow K\rightarrow Y\rightarrow\overline{Y}\rightarrow0$ in $\mathscr{B}$ with $Y\in\mathcal{Y}$ and $K\in\mathcal{Y}^\bot$. It follows from $\mathcal{Y}^\bot=\mathsf{Smd}(\mathcal {Y})^\bot$ that $Y\cong\overline{Y}\bigoplus K$. Similarly, one can show that there exists an exact sequence $0\rightarrow U\rightarrow X'\rightarrow\coker\overline{\varphi}\rightarrow0$ in $\mathscr{A}$ with $ U\bigoplus{\coker\overline{\varphi}}\cong X'\in{\mathcal{X}}$ and $U\in\mathcal{X}^\bot$ by noting that ${\coker\overline{\varphi}}\in{\mathsf{Smd}(\mathcal {X})}$ and $\mathcal{X}$ is special precovering. For $T(K)\in\mathscr{A}$, there is an exact sequence $0\rightarrow T(K)\stackrel{i}{\rightarrow} N\rightarrow \overline{X}\rightarrow0$ in $\mathscr{A}$ with $N\in\mathsf{Smd}(\mathcal {X})^\bot=\mathcal {X}^\bot$ and $\overline{X}\in\mathsf{Smd}(\mathcal {X})$. Moreover, there is an exact sequence $0\rightarrow L\rightarrow X\rightarrow \overline{X}\rightarrow0$ in $\mathscr{A}$ with $X\in\mathcal{X}$ and $L\in\mathcal{X}^\bot$, and so $X\cong\overline{X}\bigoplus L$. Therefore, we have an exact sequence in $(T\downarrow\mathscr{A})$
$$0\rightarrow\left(\begin{smallmatrix}C\\D\end{smallmatrix}\right)_{\psi}{\rightarrow} \left(\begin{smallmatrix}M\\{Y}\end{smallmatrix}\right)_{\varphi}\rightarrow\left(\begin{smallmatrix}A\\B\end{smallmatrix}\right)\rightarrow0$$
with $D\cong \overline{D}\bigoplus K, Y\cong\overline{Y}\bigoplus K, C\cong \overline{C}\bigoplus U\bigoplus N \bigoplus L, M\cong \overline{M}\bigoplus U\bigoplus N\bigoplus L$,
$\psi=\left(\begin{smallmatrix}\overline{\psi}&0\\0&0\\0&i\\0&0\\\end{smallmatrix}\right)$
and $\varphi=\left(\begin{smallmatrix}\overline{\varphi}&0\\0&0\\0&i\\0&0\\\end{smallmatrix}\right)$. Clearly, $\left(\begin{smallmatrix}C\\D\end{smallmatrix}\right)_{\psi}\in\left(\begin{smallmatrix}\mathcal {X}^{\bot}\\\mathcal {Y}^{\bot}\end{smallmatrix}\right)$. Note that we have an exact sequence
$0\rightarrow\left(\begin{smallmatrix}T(Y)\\Y\end{smallmatrix}\right){\rightarrow} \left(\begin{smallmatrix}{M}\\Y\end{smallmatrix}\right)_{\varphi}\rightarrow\left(\begin{smallmatrix}X'\bigoplus X\\0\end{smallmatrix}\right)\rightarrow0$ in $(T\downarrow\mathscr{A})$. Then $\left(\begin{smallmatrix}M\\Y\end{smallmatrix}\right)_{\varphi}\in\langle{\bf p}(\mathcal {X},\mathcal {Y})\rangle$. So we obtain that $\langle{\bf p}(\mathcal {X},\mathcal {Y})\rangle$ is special precovering.

Conversely, assume that $\langle{\bf p}(\mathcal {X},\mathcal {Y})\rangle$ is special precovering,
for any $A\in\mathscr A$ and $B\in\mathscr B$. Then there is an exact sequence $\xi: 0\rightarrow\left(\begin{smallmatrix}C\\D\end{smallmatrix}\right)_\phi{\rightarrow} \left(\begin{smallmatrix}M\\{Y}\end{smallmatrix}\right)_\varphi\rightarrow\left(\begin{smallmatrix}A\\B\end{smallmatrix}\right)\rightarrow0$ in $(T\downarrow\mathscr{A})$ with
$\left(\begin{smallmatrix}M\\Y\end{smallmatrix}\right)\in\langle{\bf p}(\mathcal {X},\mathcal {Y})\rangle$ and $\left(\begin{smallmatrix}C\\D\end{smallmatrix}\right)\in\langle{\bf p}(\mathcal {X},\mathcal {Y})\rangle^\bot$. By Proposition \ref{thm:main3}, $\langle{\bf p}(\mathcal {X},\mathcal {Y})\rangle^\bot=\left(\begin{smallmatrix}\mathcal {X}^{\bot}\\\mathcal {Y}^{\bot}\end{smallmatrix}\right)$. So, the exact sequence $0\rightarrow D\rightarrow Y\rightarrow B\rightarrow0$ implies that $\mathcal {Y}$ is special precovering.  Moreover, let $B=0$ for the sequence $\xi$, then we have an exact commutative diagram $$\xymatrix{&0 \ar[d]&0\ar[d]&&\\
0\ar[r]&T(D)\ar[r]^{\simeq}\ar[d]^{\phi}&T(Y)\ar[d]^\varphi\ar[r]&T(B)=0\ar[d]\ar[r]&0\\
0\ar[r]&C\ar[r]\ar[d]&M\ar[r]\ar[d]&A\ar[r]\ar@{=}[d]&0\\
0\ar[r]&\coker\phi\ar[d]\ar[r]&\coker\ar[d]\varphi\ar[r]&A \ar[r]&0\\
&0&0.&&\\}$$
Note that $Y\in\mathcal {Y}\bigcap\mathcal {Y}^{\bot}, \coker\varphi\in\mathcal{X}$, so the middle column is split. This implies that the left column is split. If follows that  $\coker\phi\in\mathcal{X}^{\bot}$ since $C\in\mathcal{X}^{\bot}$. Thus the exact sequence $0\rightarrow\coker\phi\rightarrow\coker\varphi\rightarrow A\rightarrow0$ yields the special $\mathcal{X}$-precover of $A$.

\begin{rem}\label{rem:3.5}{\rm We can not omit the condition that $T$ is  $\mathcal {Y}$-exact in Theorem \ref{thm:special precovering}(1). For exampe, let $\Lambda$ be the $k$-algebra given by quiver $\xymatrix{\underset{3}{\bullet}\ar[r]^{\alpha}&\underset{2}{\bullet}\ar[r]^{\beta}&\underset{1}{\bullet}}$ with a relation $\beta\alpha$ and $P(i) \ ( \textrm{resp.} \ I(i), S(i))$ denotes the indecomposible projective (resp. injective, simple) module corresponding to the vertex $i$. In fact, $\Lambda=\left(\begin{smallmatrix}k & k \\0 & kA_2\\\end{smallmatrix}\right)$ with $A_2:= \xymatrix{\underset{3}{\bullet}\ar[r]^{\alpha}&\underset{2}{\bullet}}$.
Let \bc$(\mathcal {X},\mathcal {X}^{\bot})=(\mbox{Mod}k, \mbox{Mod}k)$ and $(\mathcal {Y},\mathcal {Y}^{\bot})=(\mbox {Mod}kA_2, \mathcal{I}(kA_2))$. \ec Applying $T=M\bigotimes_{kA_2}-$ to the exact sequence $0\rightarrow P(2)\stackrel{\alpha}{\rightarrow}P(3)\rightarrow S(3)\rightarrow0$ in $\mbox{Mod}kA_2$, we obtain that $T$ is not $\mathcal {Y}$-exact. Note that $\langle{\bf p}(\mbox{Mod}k,\mbox{Mod}kA_2)\rangle=\{P(3),P(2),P(1),S(3)\}$. It follows that ${\langle{\bf p}(\mbox{Mod}k,\mbox{Mod}kA_2)\rangle}^{\perp}=\{I(1),I(2),I(3)\}$. This implies that $S(2)$ has not special $\langle{\bf p}(\mbox{Mod}k,\mbox{Mod}kA_2)\rangle$-precovers. So ${\langle{\bf p}(\mbox{Mod}k,\mbox{Mod}kA_2)\rangle}$ is not special precovering in general.}\end{rem}

\section{Gorenstein Projective objects and the proof of Theorem \ref{thm:special precovering}(2)}

In this section, to get the proof of Theorem \ref{thm:special precovering}(2), we first characterize when the functor ${\bf p}:\mathscr{A}\times\mathscr{B}\rightarrow(T\downarrow\mathscr{A})$ preserves Gorenstein projective objects. It should be noted that the functor ${\bf p}$ does not preserve Gorenstein projective objects by \cite[Example 1]{Z}.

Throughout this section, $\mathscr{A}$ and $\mathscr{B}$ always have enough projective objects.

\begin{df}\label{df:5.2} The right exact functor $T:\mathscr{B}\rightarrow \mathscr{A}$ is \emph{compatible}, if the following two conditions hold:
\begin{enumerate}
\item[(C1)] $T(Q^\bullet)$ is exact for any exact sequence $Q^\bullet$ of projective objects in $\mathscr{B}$.
\item[(C2)] ${\rm Hom}_{\mathscr{A}}(P^\bullet,T(Q))$ is  exact for any complete $\mathscr{A}$-projective resolution $P^\bullet$ and any projective object $Q$ in $\mathscr{B}$.
\end{enumerate}
Moreover, $T:\mathscr{B}\rightarrow \mathscr{A}$ is called weak compatible, if it satisfies conditions (W1) and (C2), where
\begin{enumerate}
\item[(W1)] $T(Q^\bullet)$ is exact for any complete $\mathscr{B}$-projective resolution $Q^\bullet$.
\end{enumerate}
\end{df}
\begin{rem}\label{rem:4.3}{\rm (1) We note that the exact functor $e:$ $\textrm{Ch}$$(R)$ $\rightarrow$ $\Mod R$ defined in Example \ref{e exact}(4) is compatible.

(2) Let $R$ and $S$ be Artin algebras, and $\Lambda=\left(
                                                                                 \begin{smallmatrix}
                                                                                   R & M \\
                                                                                   0 & S \\
                                                                                 \end{smallmatrix}
                                                                               \right)
$ a triangular matrix algebra. If we define $T\cong M\otimes_{S}-: {\rm mod}S \ra  {\rm mod}R$, it is easy to check that $T$ is compatible if and only if $M$
 is a compatible $R$-$S$-bimodule as defined by Zhang in \cite[Definition 1.1]{Z}.

 (3) It should be noted that a weak compatible functor $T$ is not compatible in general  as the following example shows.}
 \end{rem}

\begin{exa}{\rm Let $\Lambda=kQ/I$ with quiver $\xymatrix{\underset{3}{\bullet}\ar[r]^{\beta}&\ar@(ul,ur)^{x}\underset{2}{\bullet}\ar[r]^{\alpha}
&\underset{1}{\bullet}}$ and $I=\langle x^2,\alpha x, \alpha\beta, x\beta\rangle$, that is, $\Lambda=\left(\begin{smallmatrix}R & M \\0 & S\\\end{smallmatrix}\right)$ with $R=k$ and $S=\left(\begin{smallmatrix}k[x]/\langle x^2\rangle& k\\0&k\\\end{smallmatrix}\right)$. Note that the algebra $S$ is CM-free, that is every finitely generated Gorenstein projective left $S$-module is projective, and so each complete $\mathscr{B}$-projective resolution is always split, where $\mathscr{B}$ is the category of finitely generated left $S$-modules. It follows that $T=M\bigotimes_{S}-$ must be weak compatible and the $\mathscr{B}$-projective resolution $$Q^\bullet=\cdots\rightarrow\left(\begin{smallmatrix}k[x]/\langle x^2\rangle\\0\\\end{smallmatrix}\right)\stackrel{x}{\rightarrow}\left(\begin{smallmatrix}k[y]/\langle x^2\rangle\\0\\\end{smallmatrix}\right)\stackrel{x}{\rightarrow}\left(\begin{smallmatrix}k[y]/\langle x^2\rangle\\0\\\end{smallmatrix}\right)\stackrel{x}{\rightarrow}\cdots$$ is  not complete. Note that $T(Q^\bullet)=\cdots\rightarrow k\stackrel{0}{\rightarrow}k\stackrel{0}{\rightarrow}k{\rightarrow}\cdots$ is not exact, and so $T$ is not compatible.}
\end{exa}

For any ring $R$, the  projective (resp. injective) dimension of a left $R$-module $M$ will be denoted by $\textrm{pd}_RM$ (resp. $\textrm{id}_RM$) and the  flat dimension of a right $R$-module $N$ will be denoted by $\textrm{fd}N_{R}$. The following proposition gives more examples of compatible functors.

\begin{prop}\label{dim-com}Let $M$ be an $R$-$S$-bimodule and $T=M\bigotimes_S-$. Then
\begin{enumerate}
\item If {\rm fd}$M_S$ is finite, then $T$ satisfies (C1).
\item If {\rm pd}$_RM$ is finite, then $T$ satisfies (C2).
\item If $R$ is a left noetherian ring and {\rm id}$_RM$ is finite, then $T$ satisfies (C2).
\end{enumerate}\end{prop}
\pf (1) Let $Q^\bullet=\cdots\rightarrow Q^{-1}\stackrel{d^{-1}}\rightarrow Q^0\stackrel{d^0}\rightarrow Q^1\rightarrow\cdots$ be an exact sequence of projective left $S$-modules and fd$M_S=n$. By Dimension Shifting, we have ${\rm Tor}_S^1(M, {\ker} d^i)={\rm Tor}_S^{n+1}(M, {\ker} d^{n+i})=0$ for any integer $i$. It follows that $T(Q^\bullet)$ is still exact.

(2) and (3). For any projective left $S$-module $Q$, there is an index $I$ such that $Q$ is a direct summand of $S^{(I)}$. Then $M\bigotimes_{S} Q$ is a direct summand of $M\bigotimes_{S} S^{(I)}\cong {M^{(I)}}$, and hence pd$_RM\bigotimes_{S} Q$ is finite if pd$_RM$ is finite and  id$_RM\bigotimes_{S} Q$ is finite if id$_RM$ is finite and $R$ is left noetherian.  Let $P^\bullet=\cdots\rightarrow P^{-1}\stackrel{d^{-1}}\rightarrow P^0\stackrel{d^0}\rightarrow P^1\rightarrow\cdots$ be a complete $\mathscr{A}$-projective resolution. Then (2) and (3) hold  by the isomorphism  ${\rm Ext}_{R}^1({\ker} d^i, M\bigotimes_{S} Q)\cong{\rm Ext}_{R}^{n+1}({\ker} d^{n+i}, M\bigotimes_{S} Q)$ for any integer $i$.
\epf

\begin{lem}\label{res GP} Let  $G$ be an object in $\mathscr{B}$ and $L$ an object in $\mathscr{A}$.
\begin{enumerate}
\item If ${\bf p}(0,G)$ is a Gorenstein projective object, then $G$ is Gorenstein projective.
\item If ${\bf p}(L,0)$ is a Gorenstein projective object, then $L$ is Gorenstein projective.
\end{enumerate}
\end{lem}

\pf (1) Let ${\bf p}(0,G)$ be a Gorenstein projective object. Then there is a complete $(T\downarrow\mathscr{A})$-projective resolution $${{\bf p}(P^\bullet,Q^\bullet)}=\cdots\rightarrow{\bf p}(P^{-1},Q^{-1}){\rightarrow} {\bf p}(P^{0},Q^{0})\rightarrow
{\bf p}(P^{1},Q^{1})\rightarrow\cdots$$ with $Z^0({{\bf p}(P^\bullet,Q^\bullet)})={\bf p}(0,G)$.  Then  ${{\bf p}(P^\bullet,Q^\bullet)}$ is ${\rm Hom}_{(T\downarrow\mathscr{A})}(-,{\bf p}(0,Q))$ exact for any projective object $Q$ in $\mathscr{B}$, which implies that $Q^\bullet_\mathscr{B}=\cdots\rightarrow Q^{-1}\rightarrow Q^0\rightarrow Q^1\rightarrow\cdots$ is ${\rm Hom}_{\mathscr{B}}(-,Q)$ exact since  $${\rm Hom}_{(T\downarrow\mathscr{A})}({\bf p}(P^{i},Q^{i}),{\bf p}(0,Q))\simeq{\rm Hom}_{\mathscr{A}\times\mathscr{B}}((P^{i},Q^{i}),{\bf q}({\bf p}(0,Q))){\simeq {\rm Hom}_\mathscr{A}(P^{i},T(Q))\times} {\rm Hom}_\mathscr{B}(Q^{i},Q)$$ by Remark \ref{functors}. Clearly, $G=Z^0({Q}^\bullet)$, then $G$ is Gorenstein projective.

(2) Let ${\bf p}(L,0)$ be a Gorenstein projective object. Then there is a complete $(T\downarrow\mathscr{A})$-projective resolution $${{\bf p}(P^\bullet,Q^\bullet)}=\cdots\rightarrow{\bf p}(P^{-1},Q^{-1}){\rightarrow} {\bf p}(P^{0},Q^{0})\rightarrow
{\bf p}(P^{1},Q^{1})\rightarrow\cdots$$ with $Z^0({{\bf p}(P^\bullet,Q^\bullet)})={\bf p}(L,0)$.  Then ${{\bf p}(P^\bullet,Q^\bullet)}$ is ${\rm Hom}_{(T\downarrow\mathscr{A})}(-,{{\bf p}(P,0)})$ exact for any projective object $P$ in $\mathscr{A}$, which implies that $P^\bullet=\cdots\rightarrow P^{-1}\rightarrow P^0\rightarrow P^1\rightarrow\cdots$ is ${\rm Hom}_{\mathscr{A}}(-,P)$ exact since  $${\rm Hom}_{(T\downarrow\mathscr{A})}({\bf p}(P^{i},Q^{i}),{\bf p}(P,0))\simeq{\rm Hom}_{\mathscr{A}\times\mathscr{B}}((P^{i},Q^{i}),{\bf q}({\bf p}(P,0)))={\rm Hom}_{\mathscr{A}}(P^{i},P)$$ by Remark \ref{functors}. Clearly, $L=Z^0(P^\bullet)$, then $L$ is Gorenstein projective.
\epf

\begin{lem}\label{T(G)} For the comma category  $(T\downarrow\mathscr{A})$, we have
\begin{enumerate}
\item $T$ satisfies (W1) if and only if ${\bf p}(0,G)$ is a Gorenstein projective object for any Gorenstein projective object $G$ in $\mathscr{B}$.
\item $T$ satisfies (C2) if and only if ${\bf p}(L,0)$ is a Gorenstein projective object for any Gorenstein projective object $L$ in $\mathscr{A}$.
\end{enumerate}
\end{lem}

\pf (1) $``\Leftarrow"$. Let $Q^\bullet=\cdots\rightarrow Q^{-1}\rightarrow Q^0\stackrel{d^0}\rightarrow Q^1\rightarrow\cdots$ be a complete $\mathscr{B}$-projective resolution. Then it is sufficient to show that $T$ is exact corresponding to the exact sequence $0\rightarrow G^0\stackrel{i}{\rightarrow} Q^0\rightarrow G^1\rightarrow0$ where $G^0={\ker}d^0$ and $G^1={\rm im}d^0$. Since $G^0$ is Gorenstein projective, ${\bf p}(0,G^0)$ is Gorenstein projective by hypothesis. Thus there is an exact sequence $0\rightarrow{\bf p}(0,G^0){\rightarrow} \left(\begin{smallmatrix}Q_A\\Q_B\end{smallmatrix}\right)\rightarrow\left(\begin{smallmatrix}G_A\\G_B\end{smallmatrix}\right)\rightarrow0$ in $(T\downarrow\mathscr{A})$ with $\left(\begin{smallmatrix}Q_A\\Q_B\end{smallmatrix}\right)$ projective. Since $Q_B$ is projective, we have the following commutative diagram with exact rows:
$$\xymatrix{
0\ar[r]&G^0\ar[r]\ar@{=}[d]&Q^0\ar[r]\ar[d]&G^1\ar[r]\ar[d]&0\\
    0\ar[r]&G^0\ar[r]&Q_B\ar[r]&G_B\ar[r]&0.\\
    }$$
 Therefore, we obtain the following exact commutative diagram
$$\xymatrix{
&T(G^0)\ar^{T(i)}[r]\ar@{=}[d]&T(Q^0)\ar[r]\ar[d]&T(G^1)\ar[r]\ar[d]&0\\
&T(G^0)\ar[r]\ar@{=}[d]&T(Q_B)\ar[r]\ar[d]&T(G_B)\ar[r]\ar[d]&0\\
0\ar[r]&T(G^0)\ar[r]&Q_A\ar[r]&G_A\ar[r]&0,\\
    }$$
which implies that $T(i)$ is monic, as required.

$``\Rightarrow"$. Let $G$ be a Gorenstein projective object in $\mathscr{B}$. Then there is a complete $\mathscr{B}$-projective resolution $Q^\bullet=\cdots\rightarrow Q^{-1}\rightarrow Q^0\stackrel{d^0}\rightarrow Q^1\rightarrow\cdots$ such that $G={\ker} d^0$. By hypothesis,
$${\bf p}(0,Q^\bullet)=\cdots\longrightarrow{\bf p}(0,Q^{-1}){\longrightarrow} {\bf p}(0,Q^{0})\stackrel{{\bf p}(0,d^{0})}\longrightarrow{\bf p}(0,Q^{1})\longrightarrow\cdots$$ is a projective resolution. It follows that $${\rm Hom}_{(T\downarrow\mathscr{A})}({\bf p}(0,Q^{i}),{\bf p}(P,Q))\simeq{\rm Hom}_{\mathscr{A}\times\mathscr{B}}((0,Q^{i}),{\bf q}({\bf p}(P,Q)))={\rm Hom}_{\mathscr{B}}(Q^{i},Q)$$ for any projective object ${\bf p}(P,Q)$ in $(T\downarrow\mathscr{A})$. Thus, ${\bf p}(0,Q^\bullet)$ is complete. Then  ${\bf p}(0,G)$ is a Gorenstein projective object by noting that ${\bf p}(0,G)={\ker} {\bf p}(0,d^{0})$.

(2) Let $P^\bullet=\cdots\rightarrow P^{-1}\stackrel{d^{-1}}\rightarrow P^0\stackrel{d^0}\rightarrow P^1\stackrel{d^1}\rightarrow\cdots$ be a complete $\mathscr{A}$-projective resolution. Then there exists an exact sequence $${\bf p}(P^\bullet,0)=\cdots\longrightarrow{\bf p}(P^{-1},0)\stackrel{{\bf p}(d^{-1},0)}{\longrightarrow} {\bf p}(P^{0},0)\stackrel{{\bf p}(d^{0},0)}\longrightarrow
{\bf p}(P^{1},0)\stackrel{{\bf p}(d^{1},0)}\longrightarrow\cdots,$$
in $(T\downarrow\mathscr{A})$ with each term projective.  For any projective object ${\bf p}(P,Q)$ in $(T\downarrow\mathscr{A})$,  one has $${\rm Hom}_{(T\downarrow\mathscr{A})}({\bf p}(P^\bullet,0),{\bf p}(P,Q))\simeq{\rm Hom}_{\mathscr{A}\times\mathscr{B}}((P^\bullet,0),{\bf q}({\bf p}(P,Q)))={\rm Hom}_{\mathscr{A}}(P^{\bullet},P\oplus T(Q)).$$
Consequently, we get that ${\rm Hom}_{\mathscr{A}}(P^{\bullet}, T(Q))$ is exact for any projective object $Q$ in $\mathscr{B}$ if and only if ${\bf p}(P^\bullet,0)$ is a complete $(T\downarrow\mathscr{A})$-resolution for any complete $\mathscr{A}$-resolution $P^\bullet$, and so (2) holds.
\epf

The following proposition is crucial to the proof of Theorem \ref{thm:special precovering}(2) which characterizes when the functor ${\bf p}$ preserves Gorenstein projective objects.

\begin{prop}\label{GP} Let $(T\downarrow\mathscr{A})$ be a comma category. Then $\langle{\bf p}(\mathcal {GP}_\mathscr{A},\mathcal {GP}_\mathscr{B})\rangle\subseteq\mathcal {GP}_{(T\downarrow\mathscr{A})}$ if and only if $T$ is weak compatible. Moreover, if $T$ is compatible, then
$\mathcal {GP}_{(T\downarrow\mathscr{A})}=\langle{\bf p}(\mathcal {GP}_\mathscr{A},\mathcal {GP}_\mathscr{B})\rangle$.
\end{prop}

\begin{proof} By Lemmas \ref{res GP} and \ref{T(G)}, it is sufficient to show $\mathcal {GP}_{(T\downarrow\mathscr{A})}\subseteq\langle{\bf p}(\mathcal {GP}_\mathscr{A},\mathcal {GP}_\mathscr{B})\rangle$ provided that $T$ is compatible.
For any $\left(\begin{smallmatrix}H\\G\end{smallmatrix}\right)_\varphi\in\mathcal {GP}_{(T\downarrow\mathscr{A})}$,  there is a complete $(T\downarrow\mathscr{A})$-projective resolution $${\bf p}(P^\bullet,Q^\bullet)=\cdots\longrightarrow{\bf p}(P^{-1},Q^{-1})\stackrel{{\bf p}(d^{-1},\delta^{-1})}{\longrightarrow} {\bf p}(P^{0},Q^0)\stackrel{{\bf p}(d^{0},\delta^0)}\longrightarrow
{\bf p}(P^{1},Q^1)\stackrel{{\bf p}(d^{1},\delta^1)}\longrightarrow\cdots$$ with $\left(\begin{smallmatrix}H\\G\end{smallmatrix}\right)_\varphi\in {\ker}{\bf p}(d^{0},\delta^0)$. Then $Q^\bullet$ is a $\mathscr{B}$-projective resolution, and so $T(Q^\bullet)$ is exact since $T$ is compatible. Thus we have the following exact commutative diagram:
$$\xymatrix{
0\ar[r]&T(G^i)\ar[r]\ar[d]^{\varphi^i}&T({ Q^i})\ar[r]\ar[d]^{{\left(\begin{smallmatrix}{0}\\  {1} \end{smallmatrix}\right)}}&T(G^{i+1})\ar[d]^{^{\varphi^{i+1}}}\ar[r]&0\\
    0\ar[r]&H^i\ar[r]&{P^{i}\oplus T(Q^{i})}\ar[r]&H^{i+1}\ar[r]&0,\\
    }$$ where $G^i={\ker} {\delta^i}, {\left(\begin{smallmatrix}{H^i}\\  {G^i} \end{smallmatrix}\right)_{\varphi^i}}\in {\ker}{\bf p}(d^{i},\delta^{i})$ and $\varphi^i$ is canonical induced. In particular, $G^0=G, \ H^0=H$ and $\varphi^0=\varphi$. It follows that each $\varphi^i$ is monic and there is an exact sequence $$0\rightarrow{\coker} \varphi^i\rightarrow {P^i}\rightarrow\coker\varphi^{i+1}\rightarrow0.$$
Therefore, we have an $\mathscr{A}$-projective resolution $P^\bullet=\cdots\rightarrow P^{-1}\stackrel{d^{-1}}\rightarrow P^0\stackrel{d^{0}}\rightarrow P^1\rightarrow\cdots$ with ${\ker} d^0=\coker \varphi$. For any projective object $P$ in $\mathscr{A}$, one has
$${{\rm Hom}_{(T\downarrow\mathscr{A})}({\bf p}(P^i,Q^i),{\bf p}(P,0))}\simeq{\rm Hom}_{\mathscr{A}\times\mathscr{B}}((P^{i},Q^i),{\bf q}({\bf p}(P,0)))={\rm Hom}_{\mathscr{A}}(P^i,P).$$
Note that ${{\rm Hom}_{(T\downarrow\mathscr{A})}}({\bf p}(P^{\bullet},Q^{\bullet}),{\bf p}(P,0))$ is exact. It follows that $P^\bullet$ is complete. This implies that $\coker\varphi$ is a Gorenstein projective object in $\mathscr{A}$.

Let $Q$ be a projective object in $\mathscr{B}$. Applying ${\rm Hom}_{(T\downarrow\mathscr{A})}(-,{\bf p}(0,Q))$ to the complete $(T\downarrow\mathscr{A})$-projective resolution ${\bf p}(P^\bullet,Q^\bullet)$ above, we obtain that $Q^\bullet=\cdots\rightarrow Q^{-1}\rightarrow Q^0\rightarrow Q^1\rightarrow\cdots$ is ${\rm Hom}_{\mathscr{B}}(-,Q)$ exact and the proof is similar to that of Lemma \ref{res GP}(1). So $G=Z^0({Q}^\bullet)$ is Gorenstein projective. This completes the proof.
\end{proof}

\begin{rem}{\rm
We note that Proposition \ref{GP} generalizes \cite[Theorem 1.4]{Z} and \cite[Theorem 1.1]{LZHZ}. More precisely, assume that $(T\downarrow \Mod R)$ $=$ $\Mod {\Lambda}$, where $\Lambda=\left(
                                                                                 \begin{smallmatrix}R & M \\0 & S\\\end{smallmatrix}
                                                                             \right)
$ is a triangular matrix ring. If   $R$ and $S$ are Artin algebras  and $M$ is a compatible $R$-$S$-bimodule, then Proposition \ref{GP} here is just Theorem 1.4 in  \cite{Z}. On the other hand, if $R$ and $S$ are arbitrary rings  and ${\rm pd}_RM<\infty$ and ${\rm fd} M_S<\infty$, then Proposition \ref{GP} here is just Theorem 1.1 in  \cite{LZHZ}.}
\end{rem}

Recall from \cite[Theorem 2.2]{Yang} that a complex $G^\bullet$ in $\textrm{Ch}(R)$ is Gorenstein projective if and only if $G^{m}$
is a Gorenstein projective left $R$-module for all $m\in{\mathbb{Z}}$. As a consequence of  Remark \ref{rem:4.3}(1) and Proposition \ref{GP}, we have the follow corollary.

\begin{cor}Let $(e\downarrow\mathscr{A})$ be a comma category in Example \ref{e exact}(4). Then $\left(
                                                                                        \begin{smallmatrix}
                                                                                          X \\
                                                                                          Y^\bullet \\
                                                                                        \end{smallmatrix}
                                                                                      \right)_\varphi$ is a Gorenstein projective object in $(e\downarrow\mathscr{A})$ if and only if $Y^\bullet$  is a Gorenstein projective object in ${\rm Ch}(R)$ and $\varphi:Y^{0}\to X$ is an injective $R$-morphism with a Gorenstein projective cokernel.
\end{cor}

We are now in a position to prove the second statement of the main
result, Theorem \ref{thm:special precovering}.

{\bf Proof of Theorem \ref{thm:special precovering}(2).} {\rm  Assume that $T:\mathscr{B}\rightarrow \mathscr{A}$ is a compatible functor. Note that $\mathcal{GP}_{\mathscr{B}}\bigcap {\mathcal{GP}_{\mathscr{B}}}^{\bot}$ is the class of projective objects in $\mathscr{B}$. It follows from  (C2) in Definition \ref{df:5.2} that $T(\mathcal{GP}_{\mathscr{B}}\bigcap {\mathcal{GP}_{\mathscr{B}}}^{\bot})\subseteq{{\mathcal{GP}_{\mathscr{A}}}^{\bot}}$. By Theorem \ref{thm:special precovering}(1) and Propositon \ref{GP}, it suffices to show that  $T$ is $\mathcal{GP}_{\mathscr{B}}$-exact.

In fact, for any exact sequence $0\ra A \ra B\ra Y\ra 0$  in $\mathscr{B}$ with $Y\in{\mathcal{GP}_{\mathscr{B}}}$, there is an exact sequence $0\ra K_1\ra P_0\ra Y\ra 0$ in $\mathscr{B}$ such that $K_1$ is a Gorenstein object and $P_0$ is a projective object. We choose an exact sequence $0\ra L_1\ra Q_0\ra A\ra 0$ in $\mathscr{B}$ with $Q_0$ a projective object. Thus we have the following exact commutative diagram:

$$\xymatrix{&0\ar[d]&0\ar[d]&0\ar[d]&\\
0\ar[r]&L_1\ar[r]\ar[d]&Z_{1}\ar[r]\ar[d]&K_1\ar[r]\ar[d]&0\\
0\ar[r]&Q_0\ar[r]\ar[d]&W_{0}\ar[r]\ar[d]&P_0\ar[r]\ar[d]&0\\
    0\ar[r]&A \ar[r]\ar[d]&B\ar[r]\ar[d]&Y \ar[r]\ar[d]&0\\
     &0&0&0,&}$$
where the second row is split. Note that $0\ra T(K_1)\ra T(P_0)\ra T(Y)\ra 0$ is an exact sequence in $\mathscr{A}$ by (C1) in Definition \ref{df:5.2}. Applying the functor $T$ to the above diagram, we have the following commatative diagram with exact rows and columns:
$$\xymatrix{&&&0\ar[d]&\\
&T(L_1)\ar[r]\ar[d]&T(Z_{1})\ar[r]\ar[d]&T(K_1)\ar[r]\ar[d]&0\\
0\ar[r]&T(Q_0)\ar[r]\ar[d]&T(W_{0})\ar[r]\ar[d]&T(P_0)\ar[r]\ar[d]&0\\
    &T(A) \ar[r]\ar[d]&T(B)\ar[r]\ar[d]&T(Y) \ar[r]\ar[d]&0\\
     &0&0&0.&}$$
By the Snake Lemma, we have that $0\ra T(A)\ra T(B)\ra T(Y)\ra 0$ is an exact sequence in $\mathscr{A}$. This completes the proof.}\hfill$\Box$

The following example shows one can get many rings which are not necessary coherent such that the class of Gorenstein projective modules is special precovering over them.

\begin{exa}{\rm Let $S$ be a Gorenstein ring and $R=\left(\begin{smallmatrix} S & S^{(I)} \\ 0 & S\\\end{smallmatrix}\right)$. By Corollary \ref{GP precover}, $\mathcal{GP}(R)$ is a special precovering class. Note that $M=\left(\begin{smallmatrix} 0 \\ S\\\end{smallmatrix}\right)$ is a projective right $S$-module and an injective left $R$-module. Applying Corollary \ref{GP precover} again, we get that $\mathcal{GP}(\Lambda)$ is a special precovering class in $\Mod \Lambda$ for $\Lambda=\left(\begin{smallmatrix} R & M \\ 0 & S\\\end{smallmatrix}\right)$. We set $a=\left(\begin{smallmatrix} 0 & 0&0 \\ 0 &0&1_{S} \\
0 &0&0\\\end{smallmatrix}\right)$ and $I=\{x\in{\Lambda} \ : xa=0\}$. It follows that $\left(\begin{smallmatrix} 0 & S^{(I)}&0 \\ 0 &0&0 \\
0 &0&0\\\end{smallmatrix}\right)$ $\subseteq$ $I$, and hence $I$ is not finitely generated. So $\Lambda$ is not left coherent by \cite[Corollary 4.60]{Lam}.}
\end{exa}

\begin{rem} {\rm Let $\mathscr{A}$ and $\mathscr{B}$ be abelian categories. If $F:\mathscr{A}\rightarrow\mathscr{B}$ is a left exact functor, then there is also a comma category, denoted by $(\mathscr{B}\downarrow F)$ (see \cite{FGR,Marmaridis}). Thus we have the functor ${\bf h}:\mathscr{A}\times\mathscr{B}\rightarrow(\mathscr{B}\downarrow F)$ via
${\bf h}(A,B)=\left(\begin{smallmatrix}A\\ F(A)\oplus B\end{smallmatrix}\right)_{(1, 0)}$ and ${\bf h}(a,b)=\left(\begin{smallmatrix}a\\ F(a)\oplus b\end{smallmatrix}\right)$, where $(A,B)$ is an object in $\mathscr{A}\times\mathscr{B}$ and $(a,b)$ is a morphism in $\mathscr{A}\times\mathscr{B}$. Dually, one can also study when complete hereditary cotorsion pairs in abelian categories $\mathscr{A}$ and $\mathscr{B}$ can induce complete hereditary cotorsion pairs in $(\mathscr{B}\downarrow F)$, and characterize when special preenveloping classes in abelian categories $\mathscr{A}$ and $\mathscr{B}$ can induce special preenveloping classes in $(\mathscr{B}\downarrow F)$ . All the results concerning the comma category  $(T\downarrow\mathscr{A})$ have their
 counterparts by using the comma category induced by left exact functors.}
\end{rem}
\bigskip \centerline {\bf ACKNOWLEDGEMENTS}
\bigskip
This research was partially supported by NSFC (11671069,11771212), Zhejiang Provincial Natural Science Foundation of China (LY18A010032), Qing Lan Project of Jiangsu Province and Jiangsu Government Scholarship for Overseas Studies (JS-2019-328).  Part of the work was done during a visit of the first author to Charles University in Prague with a support by Jiangsu Government Scholarship. He thanks Professors Jan Trlifaj and Jan $\check{\rm{S}}$t'ov\'{\i}$\check{\rm{c}}$ek, and the
faculty of Department of Algebra for their hospitality. The authors would like to thank Professor Changchang Xi for helpful discussions on parts of this article.
The authors are grateful to the
referees for reading the paper carefully and for many suggestions on mathematics and English expressions.

\end{document}